\documentclass{amsart}%
\usepackage{amsmath}
\usepackage{amsfonts}
\usepackage{amssymb}
\usepackage{graphicx}%
\providecommand{\U}[1]{\protect\rule{.1in}{.1in}}
\newtheorem{theorem}{Theorem}
\theoremstyle{plain}

\newtheorem{conjecture}{Conjecture}
\newtheorem{corollary}{Corollary}

\newtheorem{lemma}{Lemma}

\newtheorem{proposition}{Proposition}
\newtheorem{remark}{Remark}

\numberwithin{equation}{section}
\begin{document}
\title[Multivariate generating functions]{Multivariate generating functions built of Chebyshev polynomials and some of
its applications and generalizations.}
\author{Pawe\l \ J. Szab\l owski}
\address{Emeritus in Department of Mathematics and Information Sciences,\\
Warsaw University of Technology\\
ul Koszykowa 75, 00-662 Warsaw, Poland}
\email{pawel.szablowski@gmail.com}
\thanks{The author is grateful to the unknown referee for his detailed and in-depth
remarks and suggestions.}
\date{January, 2018}
\keywords{multivariate generating functions, Kibble-Slepian formula, Chebyshev
polynomials, q-Hermite polynomials, inversion of Poisson-Mehler formula}

\begin{abstract}
We sum multivariate generating functions composed of products of Chebyshev
polynomials of the first and the second kind. That is, we find closed forms of
expressions of the type $\sum_{j\geq0}\rho^{j}\prod_{m=1}^{k}T_{j+t_{m}}%
(x_{m})\prod_{m=k+1}^{n+k}U_{j+t_{m}}(x_{m}),$ for different integers $t_{m},$
$m=1,...,n+k.$ We also find a Kibble-Slepian formula of $n$ variables with
Hermite polynomials replaced by Chebyshev polynomials of the first or the
second kind. In all the considered cases, the obtained closed forms are
rational functions with positive denominators. We show how to apply the
obtained results to integrate some rational funtions or sum some related
series of Chebyshev polynomials. We hope that the obtained formulae will be
useful in the so-called free probability. We expect also that the obtained
results should inspire further research and generalizations. In particular,
that, following methods presented in this paper, one would be able to obtain
similar formulae for the so-called $q-$Hermite polynomials. Since the
Chebyshev polynomials of the second kind considered here are the $q$-Hermite
polynomials for $q=0$. We have applied these methods in the one- and
two-dimensional cases and were able to obtain nontrivial identities concerning
$q-$Hermite polynomials.

\end{abstract}
\maketitle

\section{Introduction}

In this work we obtain closed forms of the following expressions:

Case I. The multivariate generating functions:%
\begin{equation}
\chi_{k,n}^{(t_{1},...,t_{k+n})}(x_{1},...,x_{n+k}|\rho)=\sum_{j\geq0}\rho
^{j}\prod_{m=1}^{k}T_{j+t_{m}}(x_{m})\prod_{m=k+1}^{n+k}U_{j+t_{m}}(x_{m}),
\label{_ktnu}%
\end{equation}
where $\left\vert t_{m}\right\vert ,k,n\in\{0,1,...\},$ $k+n\geq1,$
$\left\vert \rho\right\vert <1,$ $\left\vert x_{m}\right\vert \leq1$ and
$T_{j},U_{j}$ denote $j-$th Chebyshev polynomials respectively of the first
and second kind.

Case II. The so-called Kibble--Slepian formula for Chebyshev polynomials i.e.
closed forms of the expressions:%
\begin{align}
f_{T}(\mathbf{x}|K_{n})  &  =\sum_{S}(\prod_{1\leq i<j\leq n}\left(  \rho
_{ij}\right)  ^{s_{ij}})\prod_{m=1}^{n}T_{\sigma_{m}}(x_{m}),\label{_t}\\
f_{U}(\mathbf{x}|K_{n})  &  =\sum_{S}(\prod_{1\leq i<j\leq n}\left(  \rho
_{ij}\right)  ^{s_{ij}})\prod_{m=1}^{n}U_{\sigma_{m}}(x_{m}), \label{_u}%
\end{align}
where $\mathbf{x\allowbreak=\allowbreak(}x_{1},...,x_{n})$. $K_{n}$ denotes
the symmetric, non-singular, $n\times n$ matrix with ones on its diagonal and
with $\rho_{ij}$ as its non-diagonal $ij-th$ entry. $\sum_{S}$ denotes
summation over all $n(n-1)/2$ non-diagonal entries of a symmetric $n\times n-$
matrix $S_{n}$ with zeros on the main diagonal and entries $s_{ij}$ being
nonnegative integers, while $\sigma_{m}$ is the sum of the entries $s_{ij}$
along the $m-th$ row of the matrix $S_{n}.$

We will show that in the case I. all functions $\chi_{k,n}$ are rational with
common denominator $w_{n+k}(x_{1},...,x_{k+n}|\rho)$ which is a symmetric
polynomial in $x_{1},...,x_{n+k}$ of degree $2^{n+k-1}$ as well as in $\rho$
of degree $2^{n+k}$ defined recursively by (\ref{rek}).

In case II. both functions $f_{T}(\mathbf{x}|K_{n})$ and $f_{U}(\mathbf{x}%
|K_{n})$ are rational with the same denominator
\begin{equation}
V_{n}(\mathbf{x|}K_{n})=\prod_{j=1}^{n-1}\prod_{k=j+1}^{n}w_{2}(x_{k}%
,x_{j}|\rho_{kj}), \label{mianKib}%
\end{equation}
where $w_{2}$ is defined by (\ref{w2}), below.

The fact that these functions are rational, is not very surprising, given the
fact that Chebyshev polynomials could be expressed by the trigonometric
functions and the fact that by the Euler formulae the series (\ref{_ktnu}),
(\ref{_t}) and (\ref{_u}) are sums of some geometric series. However, to get
the exact forms of the denominators and especially the numerators, is nontrivial.

Both statements will be proved in the sequel. The first one in the Section
\ref{one} and the second in the Section \ref{kibb}.

Chebyshev polynomials of the second kind (that are orthogonal with respect to
the semicircle distribution) have played a similar role in the rapidly
recently developing "free probability", as the Hermite polynomials (that are
orthogonal with respect to the normal distribution) play in classical
probability. This is so because the central role in the free probability is
played by the semicircle distribution, while in the classical one the central
role is played by the normal distribution. Hence the results presented below
are of significance for the free probability theory.

The possible other applications of the results of the paper can, for example,
help in the following:

\begin{enumerate}
\item To simplify calculations of some of the multiple integrals of the form
\[
\underset{k~fold}{\int...\int}\frac{v_{m}(x_{1},...,x_{n}|\mathbf{p})}%
{\Omega_{n}(x_{1},...,x_{n}|\mathbf{p})}\prod_{j=1}^{k}(1-x_{i}^{2})^{m_{j}%
/2}dx_{1}...dx_{k},
\]
where $v_{m}$ denotes some polynomial in variables $x_{1},...,x_{n}$ and
numbers $m_{j}\allowbreak\in\{-1,1\}$, $\mathbf{p}$\textbf{\ }denotes a set of
parameters. Thus, this set might be different in cases I. or II. $\Omega_{n}%
$\textbf{\ }is equal to $w_{n}$ in the case I, (see iterative formula
(\ref{rek})) or $V_{n}$ in the case II (see formula (\ref{mianKib})). This is
based on the observation that the closed forms in Case I and Case II are the
rational functions with the denominators of the form $\Omega_{n}$ while the
numerators are, depending on the case and on numbers $t_{m},$ $m\allowbreak
=\allowbreak1,\ldots,n,$ polynomials of degree at most $\sum_{m=1}^{n}%
(t_{m}+1).$ For example, for $n\allowbreak=\allowbreak2$ see Proposition
\ref{2wym}. Hence, one could imagine expanding $\frac{v_{m}(x_{1}%
,...,x_{n}|\mathbf{p})}{\Omega_{n}(x_{1},...,x_{n}|\mathbf{p})}$ into the
linear combinations of the series of the forms (\ref{_ktnu}), (\ref{_t}) or
(\ref{_u}) depending on the cases considered Case I or Case II. Now notice
that having an absolute uniform convergence of the appropriate series
($\left\vert \rho\right\vert ,$ $\left\vert \varrho_{ij}\right\vert <1$ and
$\left\vert T_{i}(x)\right\vert ,|U_{i}(x)|\leq i,$ $\left\vert x\right\vert
\leq1,i\geq0$) one can perform integrations of each summand separately, which
is very easy.

Below we present a few examples illustrating this idea. In the first three of
these examples we will use the fact that following Proposition \ref{2wym},
iii), the numerators of the functions $\chi_{0,2}^{0,0}(x,y,\rho)$ and
$\chi_{0,2}^{2,0}(x,y,\rho)$ are equal respectively
\[
1-\rho^{2}\text{ and }(4x^{2}-4xy-1+\rho^{2}).\text{ }%
\]
Thus for $\left\vert x\right\vert ,\left\vert y\right\vert \leq1$ and
$\left\vert \rho\right\vert <1$ we get%
\begin{equation}
\int_{-1}^{1}\frac{2(1-\rho^{2})\sqrt{1-y^{2}}dy}{\pi((1-\rho^{2})^{2}%
-4xy\rho(1+\rho^{2})+4\rho^{2}(x^{2}+y^{2}))}=1, \label{E0}%
\end{equation}
\begin{equation}
\int_{-1}^{1}\frac{2(4x^{2}-4xy-1+\rho^{2})\sqrt{1-y^{2}}dy}{\pi((1-\rho
^{2})^{2}-4xy\rho(1+\rho^{2})+4\rho^{2}(x^{2}+y^{2}))}=4x^{2}-1, \label{E1}%
\end{equation}
since $U_{2}(x)\allowbreak=\allowbreak4x^{2}-1.$ In the next example we use
the (\ref{00}) to sum
\begin{equation}
\sum_{j\geq0}\rho^{2j}U_{2j}(x)=\chi_{0,2}^{0,0}(x,0,i\rho)=\frac{1+\rho^{2}%
}{(1+\rho^{2})^{2}-4\rho^{2}x^{2}} \label{oddU}%
\end{equation}
and then (\ref{IU}) and the form of $\chi_{0,2}^{2,0}(x,y,\rho)$ to get the
following result :%

\begin{equation}
\int_{-1}^{1}\frac{(4x^{2}-4xy-1+\rho^{2})dy}{\pi\sqrt{1-y^{2}}((1-\rho
^{2})^{2}-4xy\rho(1+\rho^{2})+4\rho^{2}(x^{2}+y^{2}))}=\frac{4x^{2}-1-\rho
^{2}}{(1+\rho^{2})^{2}-4x^{2}\rho^{2}}. \label{E2}%
\end{equation}
In the example below, we used the fact that, following Proposition \ref{2wym},
iv), the numerator of the function $\chi_{1,1}^{1,0}(y,x,\rho)$ is equal to
$(y(1+\rho^{2})-2\rho x).$ Hence taking into account (\ref{IT}) and the fact
that $U_{1}(x)\allowbreak=\allowbreak2x$ we get:%

\begin{equation}
\int_{-1}^{1}\frac{2(y(1+\rho^{2})-2\rho x)\sqrt{1-y^{2}}dy}{\pi((1-\rho
^{2})^{2}-4xy\rho(1+\rho^{2})+4\rho^{2}(x^{2}+y^{2}))}=-\rho x. \label{E3}%
\end{equation}

The following two example exploit the form Corollary \ref{3wym},ii) and either
(\ref{oU})%

\begin{equation}
\frac{2}{\pi}\int_{-1}^{1}\frac{(1+\rho^{2})^{3}+16\rho^{3}xyz-4\rho
^{2}(1+\rho^{2})(x^{2}+y^{2}+z^{2})}{w_{3}(x,y,z|\rho)}\sqrt{1-z^{2}}dz=1,
\label{E4}%
\end{equation}
or (\ref{IU}) and then, of course, one of the formulae given in Proposition
\ref{2wym} to sum the obtained infinite series:%

\begin{gather}
\frac{1}{\pi}\int_{-1}^{1}\frac{(1+\rho^{2})^{3}+16\rho^{3}xyz-4\rho
^{2}(1+\rho^{2})(x^{2}+y^{2}+z^{2})}{\sqrt{1-z^{2}}w_{3}(x,y,z|\rho
)}dz\label{E5}\\
=\frac{(1-\rho^{2})^{3}+4\rho^{2}(1-\rho^{2})(x^{2}+y^{2})}{(1-\rho^{2}%
)^{4}+16\rho^{4}(x^{4}+y^{4})+8\rho^{2}(1-\rho^{2})^{2}(x^{2}+y^{2}%
)-16\rho^{2}(1+\rho^{4})x^{2}y^{2}},\nonumber
\end{gather}

we have here $w_{3}(x,y,z|\rho)$ is given by (\ref{w3}).

\item To derive several expansions of the type (\ref{_u}) and (\ref{_t}) for
the special choices of the parameters $x_{j}$. To illustrate this idea we have
the following examples:%
\begin{equation}
\sum_{j=0}^{\infty}(j+1)\rho^{j}U_{j}(x)U_{j}(y)=\frac{(1+\rho^{2})(1-\rho
^{2})^{2}-4\rho^{2}(1+\rho^{2})(x^{2}+y^{2})+16\rho^{3}xy}{((1-\rho^{2}%
)^{2}-4xy\rho(1+\rho^{2})+4\rho^{2}(x^{2}+y^{2}))^{2}}, \label{EE1}%
\end{equation}
\begin{gather}
\sum_{j\geq0}t^{j}T_{2j+1}(x)T_{2j+1}(y)=\label{EE2}\\
\frac{(1-t)xy(1+6t+t^{2}-4t(x^{2}+y^{2}))}{(1-t)^{4}+8t(1-t)^{2}(x^{2}%
+y^{2})-16t(1+t^{2})x^{2}y^{2}+16t^{2}(x^{4}+y^{4})}.\nonumber
\end{gather}
To get these identities we used formualae given in (\ref{00}), (\ref{T11}),
(\ref{U11}) as well as in Corollary \ref{3wym},

\item To obtain families of multivariate distributions in $\mathbb{R}^{n}$
with compact support of the form:%
\[
f_{n}(x_{1},...,x_{n})=\frac{p_{m}(x_{1},...,x_{n}|\mathbf{p})}{\Omega
_{n}(x_{1},...,x_{n}|\mathbf{p})}\prod_{j=1}^{n}(1-x_{i}^{2})^{m_{j}/2},
\]
where polynomial $p_{m}$ can depend on many parameters, can have any degree,
but must me positive on $\mathbf{S\allowbreak=\allowbreak}[-1,1]^{n}$ and such
that $f_{n}$ integrates to $1$ on $\mathbf{S,}$ indices $m_{j}\in\{-1,1\}.$
\end{enumerate}

There is one more reason for which the results are important. Namely, the
Chebyshev polynomials of the second kind are, as stated above, identical with
the so-called $q-$Hermite polynomials for $q\allowbreak=\allowbreak0$. Thus
the results of the paper can be an inspiration to obtain similar results for
the $q-$Hermite polynomials. All these ideas are explained and made more
precise in the sequence of observations, remarks, hypothesis and conjectures
presented in Section \ref{gen}.

An interesting, nontrivial example of an application of the method presented
in Theorem \ref{main} applied to the well-known cases and leading to the
non-obvious identities like the ones shown by (\ref{id2wym}), (\ref{sumk}) and
(\ref{00k}) is presented in Subsection \ref{Ident}.

The paper is organized as follows. In the next section we present some
elementary observations, we recall the basic properties of Chebyshev
polynomials as well as we prove some important auxiliary results. The main
results of the paper are presented in the two successive Sections \ref{one}
and \ref{kibb} presenting respectively closed forms of the one-parameter
multivariate generating functions and the closed form of the analogue of
Kibble--Slepian formula. The next Section \ref{gen} presents generalization,
observations, conjectures and examples. Finally the last Section \ref{dow}
contains longer proofs.

\section{Auxiliary results and elementary observations\label{pom}}

Let us recall (following \cite{Mason2003}), the definitions of the Chebyshev
polynomials:
\begin{equation}
U_{n}(\cos(\alpha))\allowbreak=\allowbreak\sin((n+1)\alpha)/\sin(\alpha)\text{
and }T_{n}(\cos(\alpha))\allowbreak=\allowbreak\cos(n\alpha) \label{Czebysz}%
\end{equation}
and the orthogonality relations they satisfy:
\begin{align}
\int_{-1}^{1}T_{i}(x)T_{j}(x)\frac{1}{\pi\sqrt{1-x^{2}}}dx  &  =\left\{
\begin{array}
[c]{ccc}%
0 & if & i\neq j\\
1/2 & if & i=j\neq0\\
1 & if & i=j=0
\end{array}
\right.  ,\label{oT}\\
\int_{-1}^{1}U_{i}(x)U_{j}(x)\frac{2}{\pi}\sqrt{1-x^{2}}dx\allowbreak &
=\allowbreak\left\{
\begin{array}
[c]{ccc}%
0 & if & i\neq j\\
1 & if & i=j
\end{array}
\right.  . \label{oU}%
\end{align}

We have also some simple properties of Chebyshev polynomials that were useful
in obtaining examples (\ref{E1}-\ref{E5}) and (\ref{EE1},\ref{EE2}):%
\begin{equation}
T_{j}(0)\allowbreak=\allowbreak U_{j}(0)=\left\{
\begin{array}
[c]{ccc}%
0 & if & j\text{ is odd}\\
(-1)^{j/2} & if & j\text{ is even}%
\end{array}
\right.  , \label{00}%
\end{equation}

\begin{gather}
T_{i}(1)=1,T_{j}(-1)=(-1)^{j-2\left\lfloor j/2\right\rfloor },\label{T11}\\
U_{j}(\pm1)=\pm(j+1), \label{U11}%
\end{gather}
for $j\geq0,$%

\begin{equation}
\int_{-1}^{1}T_{j}(x)\frac{2\sqrt{1-x^{2}}}{\pi}dx=\left\{
\begin{array}
[c]{ccc}%
1 & if & j=0\\
-1/2 & if & j=2\\
0 & if & j\notin\{0,2\}
\end{array}
\right.  , \label{IT}%
\end{equation}
and%
\begin{equation}
\int_{-1}^{1}U_{j}(x)\frac{1}{\pi\sqrt{1-x^{2}}}dx=\left\{
\begin{array}
[c]{ccc}%
0 & if & j\text{ is odd}\\
1 & if & j\text{ is even}%
\end{array}
\right.  . \label{IU}%
\end{equation}

In the sequel, if all integer parameters $t_{1},...,t_{n+k}$ will be equal to
zero, then they will be dropped from function $\chi$. Notice also that the
functions $\chi$ are known for $n\allowbreak=\allowbreak1$ and $n\allowbreak
=\allowbreak2$ and $t_{1}\allowbreak=\allowbreak0,$ $t_{2}\allowbreak
=\allowbreak0$. By (\ref{_ktnu}) we have:%
\begin{gather}
\chi_{0,1}(x|\rho)=\frac{1}{w_{1}(x|\rho)};\chi_{1,0}(x|\rho)=\frac{1-\rho
x}{w_{1}(x|\rho)},\label{_1}\\
\chi_{0,2}(x,y|\rho)\allowbreak=\allowbreak\sum_{n\geq0}\rho^{n}U_{n}%
(x)U_{n}(y)=\frac{1-\rho^{2}}{w_{2}(x,y|\rho)},\label{_2}\\
\chi_{2,0}(x,y|\rho)=\sum_{n\geq0}\rho^{n}T_{n}(x)T_{n}(y)=\frac{1-\rho
^{2}+2\rho^{2}\left(  x^{2}+y^{2}\right)  -\left(  \rho^{2}+3\right)  \rho
xy}{w_{2}(x,y|\rho)},\label{_3}\\
\chi_{1,1}(x,y|\rho)=\sum_{n\geq0}\rho^{n}U_{n}(x)T_{n}(y)=\frac{1-\rho
^{2}-2\rho xy+2\rho^{2}y^{2}\allowbreak}{w_{2}(x,y|\rho)}, \label{_4}%
\end{gather}
where:
\begin{align}
w_{1}(x|\rho)  &  =1-2\rho x+\rho^{2},\label{w1}\\
w_{2}(x,y|\rho)  &  =(1-\rho^{2})^{2}-4xy\rho(1+\rho^{2})+4\rho^{2}%
(x^{2}+y^{2}). \label{w2}%
\end{align}

Notice also that both $\chi_{2,0}$ and $\chi_{0,2}$ are positive on
$[-1,1]\times\lbrack-1,1].$ The formulae in (\ref{_1}) are well known within
e.g. theory of Poisson kernel. The formula in (\ref{_2}) it is famous
Poisson-Mehler formula for $q-$Hermite polynomials where we set $q=0.$ Both
can be found in \cite{Mason2003}. The second formula in (\ref{_3}) and in
(\ref{_4}) have been recently obtained in \cite{Szab-Cheb}.

To calculate the functions $\chi_{k,n}^{(t_{1},...,t_{k+n})}$ we need the
following auxiliary results. They are very simple, based on the elementary
properties of the trigonometric functions. We present them for the sake of the
completeness of the paper. We have:

\begin{proposition}%
\begin{equation}
w_{1}(\cos(\alpha+\beta)|\rho)w_{1}(\cos(\alpha-\beta)|\rho)\allowbreak
=w_{2}(\cos(\alpha),\cos(\beta)|\rho). \label{pro2}%
\end{equation}

\end{proposition}

\begin{proof}
We have
\begin{gather*}
(1-2\rho\cos(\alpha+\beta)+\rho^{2})((1-2\rho\cos(\alpha-\beta)+\rho
^{2})\allowbreak=\allowbreak\\
(1+\rho^{2})^{2}-2\rho(1+\rho^{2})(\cos(\alpha+\beta)+\cos(\alpha
-\beta))+4\rho^{2}\cos(\alpha+\beta)\cos(\alpha-\beta).
\end{gather*}
Now recall that $\cos(\alpha+\beta)+\cos(\alpha-\beta)\allowbreak
=\allowbreak2\cos(\alpha)\cos(\beta)$ and $\cos(\alpha+\beta)\cos(\alpha
-\beta)\allowbreak=\allowbreak\cos^{2}\alpha\allowbreak+\allowbreak\cos
^{2}\beta\allowbreak-\allowbreak1.$
\end{proof}

\begin{proposition}
\label{ilocz}%
\begin{gather}
\prod_{j=1}^{k}\cos(\alpha_{j})\allowbreak=\allowbreak\frac{1}{2^{n}}%
\sum_{i_{1}\in\{-1,1\}}...\sum_{i_{k}\in\{-1,1\}}\cos(\sum_{l=1}^{k}%
i_{l}\alpha_{l}),\label{cos}\\
\prod_{j=1}^{n}\sin(\alpha_{j})\allowbreak\prod_{j=n+1}^{n+k}\cos(\alpha
_{j})=\allowbreak\nonumber\\
\left\{
\begin{array}
[c]{ccc}%
\begin{array}
[c]{c}%
(-1)^{(n+1)/2}\frac{1}{2^{n+k}}\sum_{i_{1}\in\{-1,1\}}...\sum_{i_{n+k}%
\in\{-1,1\}}\\
(-1)^{\sum_{l=1}^{n}(i_{l}+1)/2}\sin(\sum_{l=1}^{n+k}i_{l}\alpha_{l})
\end{array}
& if & n\text{ is odd}\\%
\begin{array}
[c]{c}%
(-1)^{n/2}\frac{1}{2^{n+k}}\sum_{i_{1}\in\{-1,1\}}...\sum_{i_{n+k}\in
\{-1,1\}}\\
(-1)^{\sum_{l=1}^{n}(i_{l}+1)/2}\cos(\sum_{l=1}^{n+k}i_{l}\alpha_{l})
\end{array}
& if & n\text{ is even}%
\end{array}
\right.  . \label{sin}%
\end{gather}

\end{proposition}

\begin{proof}
See section \ref{dow}.
\end{proof}

\begin{lemma}
\label{aux1}Let us take $n\in\mathbb{N},$ $\left\vert \rho_{i}\right\vert <1,$
$\alpha_{i}\in\mathbb{R},$ $i\allowbreak\in S_{n}\allowbreak=\allowbreak
\{1,...,n\}.$ Let $M_{i,n}$ denote a subset of the set $S_{n}$ containing $i$
elements. Let us denote by $\sum_{M_{i,n}\subseteq S_{n}}$ summation over all
$M_{i,n}$ contained in $S_{n}.$ We have:%
\begin{gather}
\sum_{k_{1}\geq0}...\sum_{k_{n}\geq0}(\prod_{i=1}^{n}\rho_{i}^{k_{i}}%
)\cos(\beta+\sum_{i=1}^{n}k_{i}\alpha_{i})\allowbreak=\label{ksc}\\
\allowbreak\frac{\sum_{j=0}^{n}(-1)^{j}\sum_{M_{j,n}\subseteq S_{n}}%
(\prod_{k\in M_{j,n}}\rho_{k})\cos(\beta-\sum_{k\in M_{j,n}}\alpha_{k})}%
{\prod_{i=1}^{n}(1+\rho_{i}^{2}-2\rho_{i}\cos(\alpha_{i}))},\nonumber\\
\sum_{k_{1}\geq0}...\sum_{k_{n}\geq0}(\prod_{i=1}^{n}\rho_{i}^{k_{i}}%
)\sin(\beta+\sum_{i=1}^{n}k_{i}\alpha_{i})=\label{kss}\\
\frac{\sum_{j=0}^{n}(-1)^{j}\sum_{M_{j,n}\subseteq S_{n}}(\prod_{k\in M_{j,n}%
}\rho_{k})\sin(\beta-\sum_{k\in M_{j,n}}\alpha_{k})}{\prod_{i=1}^{n}%
(1+\rho_{i}^{2}-2\rho_{i}\cos(\alpha_{i}))}.\nonumber
\end{gather}

\end{lemma}

\begin{proof}
See section \ref{dow}.
\end{proof}

We will also need the following almost trivial special cases of formulae
(\ref{ksc}) and (\ref{kss}). We will formulate them as corollary.

\begin{corollary}
\label{suma}For all $\left\vert \rho\right\vert <1$ we have
\begin{align}
\sum_{n\geq0}\rho^{n}\sin(n\alpha+\beta)  &  =(\sin(\beta)-\rho\sin
(\beta-\alpha))/(1-2\rho\cos(\alpha)+\rho^{2}),\label{s_si}\\
\sum_{n\geq0}\rho^{n}\cos(n\alpha+\beta)  &  =(\cos(\beta)-\rho\cos
(\beta-\alpha)/(1-2\rho\cos(\alpha)+\rho^{2}). \label{s_g_c}%
\end{align}

\end{corollary}

\begin{proof}
Set $n\allowbreak=\allowbreak1$ and $\alpha\allowbreak=\allowbreak\alpha_{1}$
(\ref{kss}) and (\ref{ksc}).
\end{proof}

\section{One parameter sums. Multivariate generating functions of Chebyshev
polynomials\label{one}}

The theorem below is obtained by very elementary methods. Given the definition
of the function $\chi_{k,n}^{(t_{1},...,t_{n+k})}(x_{1},...,x_{n+k}|\rho)$
presented by (\ref{_ktnu}) it is obvious that it must be in the form of a
rational function. Even many properties of the denominator of these functions
can be more or less deduced from the definition. However the exact forms of
the numerators of these functions are not trivial. For the sake of
completeness of the paper, we present all these trivial and nontrivial
observations in one theorem.

\begin{theorem}
\label{main}For all integers $n,k\geq0,$ $\left\vert x_{s}\right\vert
<1,t_{s}\in\mathbb{Z},$ $s\allowbreak=\allowbreak1,...,n+k,$ we have:%
\begin{equation}
\chi_{k,n}^{(t_{1},...,t_{n+k})}(x_{1},...,x_{n+k}|\rho)=\allowbreak
\frac{l_{k,n}^{(t_{1},...,t_{n+k})}(x_{1},...,x_{n+k}|\rho)}{w_{n+k}%
(x_{1},...,x_{n+k}|\rho)}, \label{formula}%
\end{equation}
where $w_{m}(x_{1},...,x_{m}|q)$ is a symmetric polynomial of degree $2^{m-1}$
in $x_{1},...,x_{m}$ and of degree $2^{m}$ in $\rho$ defined by the following
recurrence :%
\begin{gather}
w_{m+1}(x_{1},...,x_{m-1},\cos(\alpha),\cos(\beta)|\rho)=\label{rek}\\
w_{m}(x_{1},...,x_{m-1},\cos(\alpha+\beta)|\rho)w_{m}(x_{1},...,x_{m-1}%
,\cos(\alpha-\beta)|\rho),\nonumber
\end{gather}
$n\geq1$, with $w_{1}(x|q)$ given by (\ref{w1}).

$l_{n,k}^{(t_{1},\ldots,t_{n+k})}(x_{1},...,x_{n+k}|\rho)$ is another
polynomial given by the relationship:%
\begin{gather}
l_{k,n}^{(t_{1},...,t_{n+k})}(x_{1},...,x_{n+k}|\rho)=\label{diff}\\
\sum_{j=0}^{2^{n+k}-1}\rho^{j}\sum_{m=0}^{j}\frac{1}{m!}\left.  \frac{d^{m}%
}{d\rho^{m}}w_{k+n}(x_{1},...,x_{k+n}|\rho)\right\vert _{\rho=0}\nonumber\\
\times\prod_{s=1}^{k}T_{(j-m)+t_{s}}(x_{s})\prod_{s=1+k}^{n+k}U_{(j-m)+t_{s}%
}(x_{s}).
\end{gather}

\end{theorem}

\begin{proof}
See section \ref{dow}.
\end{proof}

\begin{corollary}
Theorem \ref{main} provides for free the following important set of identities
involving Chebyshev polynomial of the first and the second kind. Namely we
have: $\forall n,k\geq0:n+k\geq1,\forall t_{1},\ldots,t_{n+k}\geq0,\forall
j\geq2^{n+k},\forall(x_{1},\ldots,x_{k+n})\in(-1,1)^{n+k}$%
\begin{equation}
\sum_{m=0}^{j}\frac{1}{m!}\left.  \frac{d^{m}}{d\rho^{m}}w_{k+n}%
(x_{1},...,x_{k+n}|\rho)\right\vert _{\rho=0}\times\prod_{s=1}^{k}%
T_{(j-m)+t_{s}}(x_{s})\prod_{s=1+k}^{n+k}U_{(j-m)+t_{s}}(x_{s})=0.
\label{identities}%
\end{equation}

In particular we have for $n+k\allowbreak=\allowbreak1:$
\[
U_{k}(x)-2xU_{k+1}(x)+U_{k+2}(x)=0,
\]
which is nothing else as the well-known three-term recurrence satisfied by the
Chebyshev polynomials. However for say $k=0$ and $n\allowbreak=\allowbreak2$
we get for all $s,m\geq0$%
\begin{gather*}
-4xyU_{s}(y)U_{m}(x)+2(2x^{2}+2y^{2}-1)U_{s+1}(y)U_{m+1}(x)\\
-4xyU_{s+2}(y)U_{m+3}(x)+U_{s+3}(y)U_{m+3}(x)=0,
\end{gather*}
which is, to my knowledge, unknown.
\end{corollary}

\begin{proof}
Since $l_{k,n}^{(t_{1},...,t_{n+k})}(x_{1},...,x_{n+k}|\rho)$ is a polynomial
of degree $2^{k+n}-1$ in $\rho$ all its derivatives with respect to $\rho$ of
higher than $2^{k+n}-1$ should be equal to zero.
\end{proof}

\begin{corollary}
For $n\geq1,$ after swapping $x_{1}$ and $x_{n},$ taking $\beta\allowbreak
=\allowbreak0,$ $\cos(\alpha)\allowbreak=\allowbreak x_{2}$ we get:%
\[
w_{n}(1,...x_{n-1},x_{n}|\rho)=(w_{n-1}(x_{2},...x_{n}|\rho))^{2}.
\]
In particular
\[
w_{3}(x_{1},\cos(\alpha_{2}),\cos(\alpha_{3})|\rho)=w_{2}(x_{1},\cos
(\alpha_{3}+\alpha_{2})|\rho)w_{2}(x_{1},\cos(\alpha_{3}-\alpha_{2})|\rho),
\]
which, after replacing $\cos(\alpha_{2})$ by $x_{2}$ and $\cos(\alpha_{3})$ by
$x_{3}$ and with the help of Mathematica, yields:%
\begin{gather}
w_{3}(x_{1},x_{2},x_{3}|\rho)=16\rho^{4}(x_{1}^{4}+x_{2}^{4}+x_{3}^{4}%
)-8\rho^{2}(1+\rho^{2})^{2}(x_{1}^{2}+x_{2}^{2}+x_{3}^{2})\label{w3}\\
+16\rho^{2}(1+\rho^{4})(x_{1}^{2}x_{2}^{2}+x_{1}^{2}x_{3}^{2}+x_{2}^{2}%
x_{3}^{2})+64\rho^{4}x_{1}^{2}x_{2}^{2}x_{3}^{2}-32\rho^{3}(1+\rho^{2}%
)x_{1}x_{2}x_{3}(x_{1}^{2}+x_{2}^{2}+x_{3}^{2})\nonumber\\
-8\rho(1+\rho^{2})(1+\rho^{4}-6\rho^{2})x_{1}x_{2}x_{3}+(1+\rho^{2}%
)^{4}.\nonumber
\end{gather}

\end{corollary}

\begin{remark}
Notice that from Theorem \ref{main} we deduce that for all integers
$t_{1},...,t_{k+n}$ the ratio
\[
\frac{\chi_{k,n}^{(t_{1},...,t_{k+n})}(x_{1},...,x_{n+k}|\rho)}{\chi
_{k,n}^{(0,...,0)}(x_{1},...,x_{n+k}|\rho)}%
\]
is a rational function of arguments $x_{1},...,x_{n+k},\rho.$

Such observation for was first made by Carlitz for $k+n\allowbreak
=\allowbreak2,$ nonnegative integers $t_{1}$ and $t_{2}$ concerning the
so-called Rogers--Szeg\"{o} polynomials and two variables $x_{1}$ and $x_{2}$
in \cite{Carlitz72} (formula 1.4). Later it was generalized by Szab\l owski in
\cite{Szab5} for the so-called $q-$Hermite polynomials, also for the two
variables . Now, it turns out that for $q\allowbreak=\allowbreak0$ the
$q-$Hermite polynomials are equal to Chebyshev polynomials of the second kind,
hence one can state that so far the above-mentioned observation was known for
$k\allowbreak=\allowbreak0$ and $n\allowbreak=\allowbreak2.$ Hence we deal
with far-reaching generalization both in the number of variables as well as
for the Chebyshev polynomials of the first kind.
\end{remark}

\begin{corollary}
\label{gest}For $\left\vert x_{i}\right\vert \leq1$ and $\left\vert
\rho\right\vert <1,$ $n\geq1:$%
\begin{gather*}
\chi_{n,0}(x_{1},...,x_{n}|\rho)\geq0,\\
\underset{j\text{ fold}}{\int_{-1}^{1}...\int_{-1}^{1}(}\prod_{s=1}^{n}%
\frac{1}{\pi\sqrt{1-x_{s}^{2}}})\chi_{n,0}(x_{1},...,x_{n}|\rho)dx_{1}%
...dx_{j}\allowbreak=\allowbreak\prod_{s=j+1}^{n}\frac{1}{\pi\sqrt{1-x_{s}%
^{2}}},
\end{gather*}
for $j=1,...,n.$
\end{corollary}

\begin{proof}
For the first assertion recall that based on Theorem \ref{main} we have
\begin{gather*}
\chi_{n,0}(\cos(\alpha_{1}),...,\cos(\alpha_{n})|\rho)\allowbreak
=\allowbreak\sum_{k\geq0}\rho^{k}\prod_{j=1}^{n}T_{k}(\cos(\alpha
_{j}))\allowbreak=\\
\allowbreak\frac{1}{2^{n}}\sum_{i_{1}\in\{-1,1\}}...\sum_{i_{n}\in
\{-1,1\}}\frac{(1-\rho\cos(\sum_{k=1}^{n}i_{k}\alpha_{k}))}{(1-2\rho\cos
(\sum_{k=1}^{n}i_{k}\alpha_{k})+\rho^{2})},
\end{gather*}
which is nonnegative for all $\alpha_{i}\in\mathbb{R}$, $i\allowbreak
=\allowbreak1,...,n$ and $\left\vert \rho\right\vert <1.$ \newline The
remaining part follows directly the definition (\ref{_ktnu}) of $\chi_{n,0}$
and the properties of polynomials $T_{i}$.
\end{proof}

Let us now finish the case $n\allowbreak=\allowbreak2.$ That is let us
calculate $\chi_{2,0}^{n,m}(x,y|\rho),$ $\chi_{1,1}^{n,m}(x,y|\rho)$. The case
of $\chi_{0,2}^{n,m}(x,y|\rho)$ has been solved in e.g. \cite{SzablAW} (Lemma
3, with $q\allowbreak=\allowbreak0).$

\begin{proposition}
\label{2wym}i)
\begin{align*}
\chi_{1,0}^{m,0}(x|\rho)  &  =\sum_{i=0}^{\infty}\rho^{i}T_{i+m}%
(x)\allowbreak=\allowbreak\frac{T_{m}(x)-\rho T_{m-1}(x)}{w_{1}(x|\rho)},\\
\chi_{0,1}^{0,m}(x|\rho)  &  =\sum_{i=0}^{\infty}\rho^{i}U_{i+m}%
(x)=\frac{U_{m}(x)-\rho U_{m-1}(x)}{w_{1}(x|\rho)},
\end{align*}

ii)%
\begin{gather*}
\chi_{2,0}^{n,m}(x,y|\rho)\allowbreak=\allowbreak\sum_{k\geq0}\rho^{k}%
T_{k+n}(x)T_{k+m}(y)\allowbreak=\allowbreak\\
(T_{n}(x)T_{m}(y)(w_{2}(x,y|\rho)-\rho^{4})\\
+\rho T_{n+1}(x)T_{m+1}(y)(1-2\rho^{2}+4\rho^{2}(x^{2}+y^{2})-4\rho xy)\\
+\rho^{2}T_{n+2}(x)T_{m+2}(y)(1-4\rho xy)+\rho^{3}T_{n+3}(x)T_{m+3}%
(y))/w_{2}(x,y|\rho),
\end{gather*}

iii)
\begin{gather*}
\chi_{0,2}^{n,m}(x,y|\rho)\allowbreak=\sum_{j\geq0}\rho^{j}U_{j+n}%
(x)U_{j+m}(y)\allowbreak=\\
(U_{n}(x)U_{m}(y)(w_{2}(x,y|\rho)-\rho^{4})\\
+\rho U_{n+1}(x)U_{m+1}(y)(1-2\rho^{2}+4\rho^{2}(x^{2}+y^{2})-4\rho xy)\\
+\rho^{2}U_{n+2}(x)U_{m+2}(y)(1-4\rho xy)+\rho^{3}U_{n+3}(x)U_{m+3}%
(y))/w_{2}(x,y|\rho)
\end{gather*}
$\allowbreak$\newline

iv)
\begin{gather*}
\chi_{1,1}^{n,m}(x,y|\rho)\allowbreak=\allowbreak\sum_{j\geq0}\rho^{j}%
U_{m+j}(x)T_{n+j}(y)\allowbreak=\\
(T_{n}(y)U_{m}(x)(w_{2}(x,y|\rho)-\rho^{4})\\
+\rho T_{n+1}(y)U_{m+1}(x)(1-2\rho^{2}+4\rho^{2}(x^{2}+y^{2})-4\rho xy)\\
+\rho^{2}T_{n+2}(y)U_{m+2}(y)(1-4\rho xy)+\rho^{3}T_{n+3}(y)U_{m+3}%
(y))/w_{2}(x,y|\rho).
\end{gather*}
$\allowbreak\allowbreak\allowbreak$
\end{proposition}

\begin{proof}
We apply a formula (\ref{diff}). For i) we take $n=1$ and notice that values
of derivatives of $w_{1}$ respect to $\rho$ at $\rho=0$ are $1,$ $-2x,$ $2.$

To get ii) we notice that subsequent derivatives of $w_{2}$ with respect to
$\rho$ at $\rho=0$ are $1,$ $-4xy,$ $8x^{2}+8y^{2}-4,$ $-24xy$. Having this
and applying directly (\ref{diff}) we get certain defined formula expanded in
powers of $\rho.$ Now it takes Mathematica to get this form.

iii) and iv) We argue similarly getting expansions in powers of $\rho.$ Then
using Mathematica we try to get more friendly form.
\end{proof}

As a corollary we get formulae presented in (\ref{_2}) and (\ref{_3}) when
setting $n\allowbreak=\allowbreak m\allowbreak=\allowbreak0$ and remembering
that $T_{-i}(x)\allowbreak=\allowbreak T_{i}(x),$ $U_{-i}(x)=-U_{i-2}(x),$ for
$i\allowbreak=\allowbreak0,1,2$.

\begin{corollary}
\label{3wym}$\forall x,y,z\in\lbrack-1,1],\left\vert \rho\right\vert <1:$

i)
\begin{gather*}
\chi_{3,0}(x,y,z|\rho)=\sum_{i\geq0}\rho^{i}T_{i}(x)T_{i}(y)T_{i}%
(z)\allowbreak=\allowbreak((1+\rho^{2})^{3}\allowbreak\allowbreak
+\allowbreak8\rho^{4}\left(  x^{4}+y^{4}+z^{4}\right)  \allowbreak
+\allowbreak\allowbreak32\rho^{4}x^{2}y^{2}z^{2}\allowbreak\\
-\allowbreak2\left(  \rho^{2}+1\right)  \left(  \rho^{2}+3\right)  \rho
^{2}\left(  x^{2}+y^{2}+z^{2}\right)  \allowbreak+\allowbreak4\left(  \rho
^{4}+3\right)  \rho^{2}\left(  x^{2}y^{2}+x^{2}z^{2}+y^{2}z^{2}\right)
\allowbreak\\
-\newline\allowbreak4\left(  3\rho^{2}+5\right)  \rho^{3}xyz\left(
x^{2}+y^{2}+z^{2}\right)  \allowbreak-\allowbreak\left(  \rho^{6}-15\rho
^{4}-25\rho^{2}+7\right)  \rho xyz)\allowbreak/w_{3}(x,y,z|\rho),
\end{gather*}

ii)
\begin{gather*}
\chi_{0,3}(x,y,z|\rho)=\sum_{i\geq0}\rho^{i}U_{i}(x)U_{i}(y)U_{i}(z)=\\
((1+\rho^{2})^{3}+16\rho^{3}xyz-4\rho^{2}(1+\rho^{2})(x^{2}+y^{2}%
+z^{2}))/w_{3}(x,y,z|\rho),
\end{gather*}

iii)
\begin{gather*}
\chi_{1,2}(x,y,z|\rho)=\sum_{i\geq0}\rho^{i}T_{i}(x)U_{i}(y)U_{i}(z)=\\
(\left(  \rho^{2}+1\right)  ^{3}\allowbreak+\allowbreak8\rho^{4}%
x^{4}\allowbreak-\allowbreak16\rho^{3}x^{3}yz\allowbreak-\allowbreak2\left(
\rho^{2}+1\right)  \left(  \rho^{2}+3\right)  \rho^{2}x^{2}\allowbreak\\
\allowbreak+\allowbreak8\rho^{2}x^{2}\left(  y^{2}+z^{2}\right)
-\allowbreak4\rho\left(  5-(\rho^{2}+2)^{2}\right)  xyz\allowbreak
-\allowbreak4\left(  \rho^{2}+1\right)  \rho^{2}(y^{2}+z^{2})\allowbreak
)/w_{3}(x,y,z|\rho),
\end{gather*}

iv)
\begin{gather*}
\chi_{2,1}(x,y,z|\rho)\sum_{i\geq0}\rho^{i}T_{i}(x)T_{i}(y)U_{i}(z)=\\
(\left(  \rho^{2}+1\right)  ^{3}\allowbreak+\allowbreak8\rho^{4}\left(
x^{4}+y^{4}\right)  \allowbreak-\allowbreak2\left(  \rho^{2}+1\right)  \left(
\rho^{2}+3\right)  \rho^{2}\left(  x^{2}+y^{2}\right)  \allowbreak\\
+\allowbreak4\left(  \rho^{4}+3\right)  \rho^{2}x^{2}y^{2}\allowbreak
+\allowbreak16\rho^{4}x^{2}y^{2}z^{2}\allowbreak+\allowbreak8\rho^{2}%
z^{2}\left(  x^{2}+y^{2}\right)  \allowbreak-\allowbreak8\left(  \rho
^{2}+2\right)  \rho^{3}xyz\left(  x^{2}+y^{2}\right)  \allowbreak\\
-\allowbreak8\rho^{3}xyz^{3}\allowbreak-\allowbreak2\left(  -5\rho^{4}%
-10\rho^{2}+3\right)  \rho xyz\allowbreak-\allowbreak4\left(  \rho
^{2}+1\right)  \rho^{2}z^{2})/w_{3}(x,y,z|\rho),
\end{gather*}

where $w_{3}(x,y,z|\rho)\allowbreak$ is given by (\ref{w3}).
\end{corollary}

\begin{proof}
Again we apply formula (\ref{diff}). Besides we take $n\allowbreak
=\allowbreak3,$ $k\allowbreak=\allowbreak0$ for i), $n\allowbreak
=\allowbreak0,$ $k\allowbreak=\allowbreak3$ for ii), $n\allowbreak
=\allowbreak1,$ $k\allowbreak=\allowbreak2$ for iii) and $n\allowbreak
=\allowbreak2,$ $k\allowbreak=\allowbreak1$ for iv). Now we have to remember
that successive derivatives of $w_{3}$ with respect to $\rho$ taken at
$\rho\allowbreak=\allowbreak0$ are respectively $1,$ $-8xyz,$ $8(1\allowbreak
-\allowbreak(x^{2}+y^{2}+z^{2})\allowbreak+\allowbreak4(x^{2}y^{2}%
\allowbreak+\allowbreak x^{2}z^{2}\allowbreak+\allowbreak y^{2}z^{2})),$
$48xyz(5\allowbreak-\allowbreak4(x^{2}+y^{2}+z^{2})),$ $48(3\allowbreak
-\allowbreak8(x^{2}+y^{2}+z^{2})\allowbreak+\allowbreak8(x^{4}+y^{4}%
+z^{4})\allowbreak+\allowbreak32x^{2}y^{2}z^{2}),$ $960xyz(5\allowbreak
-\allowbreak4(x^{2}+y^{2}+z^{2})),$ $2880(1\allowbreak-\allowbreak(x^{2}%
+y^{2}+z^{2})\allowbreak+\allowbreak4(x^{2}y^{2}\allowbreak+\allowbreak
x^{2}z^{2}\allowbreak+\allowbreak y^{2}z^{2})),$ $-40320xyz.$ Then we get
certain formulae by applying directly formula (\ref{diff}). The expression are
long and not very legible. We applied Mathematica to get forms presented in
i), ii) iii) and iv).
\end{proof}

\section{Kibble--Slepian formula and related sums for Chebyshev polynomials
\label{kibb}}

Let $f_{n}(x_{1},...,x_{n}|K_{n})$ denote the density of the normal
distribution with zero expectations and non-singular covariance matrix $K_{n}$
such that $\operatorname*{var}(X_{i})=\allowbreak1$ for $i\allowbreak
=\allowbreak1,...,n,$ i.e. having $1^{\prime}s$ on the diagonal. Let
$\rho_{ij}$ denote $ij-$th entry of matrix $K_{n}.$ Consequently, the
one-dimensional marginals $f_{1}$ are given by:%
\[
f_{1}(x)\allowbreak=\allowbreak\exp(-x^{2}/2)/\sqrt{2\pi}.
\]
Let us also denote by $S_{n}$ a symmetric $n\times n$ matrix with zeros on the
diagonal and nonnegative integers as off-diagonal entries. Let us denote the
$ij-$th entry of the matrix $S_{n}$ by $s_{ij}.$ Recall that Kibble in the 40s
and Slepian in the 70s presented the following formula:%
\begin{equation}
\frac{f_{n}(x_{1},...,x_{n}|K_{n})}{\prod_{m=1}^{n}f_{1}(x_{m})}=\sum
_{S}(\prod_{1\leq i<j\leq n}\frac{\left(  \rho_{ij}\right)  ^{s_{ij}}}%
{s_{ij}!}\prod_{m=1}^{n}H_{\sigma_{m}}(x_{m})), \label{K-S}%
\end{equation}
where $H_{i}(x)$ denotes $i-th$ (so called probabilistic) Hermite polynomial
i.e. forming the orthonormal base of the space of functions square integrable
with respect to the weight $f_{1}(x)$, $\sigma_{m}\allowbreak=\allowbreak
\sum_{j=1}^{m-1}s_{jm}\allowbreak+\allowbreak\sum_{j=1+m}^{n}s_{mj},$
$\sum_{S}$ denotes, as before, summation over all $n(n-1)/2$ non-diagonal
entries of the matrix $S_{n}.$ To see more details on Kibble--Slepian formula
see e.g. recent paper by Ismail \cite{Ismal2016}. A partially successful
attempt was made by Szab\l owski in \cite{SzablKib} where for $n\allowbreak
=\allowbreak3$ the author replaced polynomials $H_{n}$ by the so called
$q-$Hermite polynomials $H_{n}(x|q)$ and $s_{ij}!$ substituted by
$[s_{ji}]_{q}!$ where $[n]_{q}\allowbreak=\allowbreak(1-q^{n})/(1-q)$ for
$\left\vert q\right\vert <1,$ $[n]_{1}\allowbreak=\allowbreak n$ and
$[n]_{q}!\allowbreak=\allowbreak\prod_{i=1}^{n}[i]_{q}$ with $[0]_{q}%
!\allowbreak=\allowbreak1.$ Taking into account that $H_{n}(x|0)\allowbreak
=\allowbreak U_{n}(x/2)$ and $[n]_{0}!\allowbreak=\allowbreak1$ we see that
(\ref{K-S}) has been generalized and summed already for other polynomials. The
intension of summing in \cite{SzablKib} was to find a generalization of the
normal distribution that has compact support. The attempt was partially
successful since also one has obtained a relatively closed form for the sum,
however the obtained sum was not positive for the suitable values of
parameters $\rho_{ij}$ and all values of parameters $\left\vert q\right\vert
<1.$

In the present paper, we are going to present closed form of the sum
(\ref{K-S}) where polynomials $H_{n}$ are replaced by Chebyshev polynomials of
both the first and second kind and $s_{ij}!$ are replaced by $1.$ This last
replacement is justified by the fact that $\left[  s_{ji}\right]
_{q}!\allowbreak=\allowbreak1$ if $q\allowbreak=\allowbreak0.$ For more
details, see publications on the so-called $q-$series and also brief
introduction at the beginning of the Section \ref{gen}, below.

In other words, we are going to find closed forms for the sums (\ref{_t}) and
(\ref{_u}), where $\mathbf{x\allowbreak}$ and $K_{n},$ used below, mean, as
before, $\mathbf{x=\allowbreak(}x_{1},...,x_{n})$ while $K_{n}$ denotes
symmetric $n\times n$ matrix with ones on its diagonal and $\rho_{ij}$ as its
$ij-th$ entry. We will assume that all $\rho^{\prime}s$ are from the segment
$(-1,1)$ and additionally that matrix $K_{n}$ is positive definite.

We have the following result:

\begin{theorem}
\label{kibble}Let us denote $\mathcal{K}_{n}\allowbreak=\allowbreak\left\{
(i,j):1\leq i<j\leq n\right\}  $, $\beta_{n,m}\allowbreak=\beta_{n,m}%
(i_{n},i_{m})\allowbreak=\allowbreak i_{n}\alpha_{n}+i_{m}\alpha_{m} $. For
$S\subseteq\mathcal{K}_{n}$ let $\rho_{S}\allowbreak=\allowbreak
\prod_{(n,m)\in S}\rho_{nm}$, $b_{S}\allowbreak=\allowbreak\sum_{(n,m)\in
S}\beta_{n,m}$, $B_{1,\ldots,n}\allowbreak=\allowbreak B(i_{1},\ldots
,i_{n})\allowbreak=\allowbreak\sum_{j=1}^{n}i_{j}\alpha_{j} $.

We have i)
\begin{gather*}
f_{T}(\cos(\alpha_{1}),...,\cos(\alpha_{n})|K_{n})\allowbreak=\\
\allowbreak\frac{1}{2^{n}}\sum_{i_{1}\in\{-1,1\}}...\sum_{i_{n}\in
\{-1,1\}}\frac{\sum_{k=0}^{n}(-1)^{k}\sum_{S_{k}\subseteq\mathcal{K}_{n}%
}^{\prime}\rho_{S_{k}}\cos(b_{S_{k}})}{\prod_{j=1}^{n}\prod_{m=j+1}%
^{n}(1-2\rho_{jm}\cos(\beta_{j,m}(i_{j},i_{m}))+\rho_{jm}^{2})},
\end{gather*}
ii) If $n$ is even then%
\begin{gather*}
f_{U}(\cos(\alpha_{1}),...,\cos(\alpha_{n})|K_{n})=\\
(-1)^{n/2}\frac{1}{2^{n}\prod_{j=1}^{n}\sin(\alpha_{j})}\sum_{i_{1}%
\in\{-1,1\}}...\sum_{i_{n+k}\in\{-1,1\}}(-1)^{\sum_{l=1}^{n}(i_{l}+1)/2}\\
\frac{\sum_{k=0}^{n}(-1)^{k}\sum_{S_{k}\subseteq\mathcal{K}_{n}}^{\prime}%
\rho_{S_{k}}\cos(B_{1,\ldots,n}-b_{S_{k}})}{\prod_{j=1}^{n}\prod_{m=j+1}%
^{n}(1-2\rho_{jm}\cos(\beta_{j,m}(i_{j},i_{m}))+\rho_{jm}^{2})},
\end{gather*}
while if $n$ is odd then
\begin{gather*}
f_{U}(\cos(\alpha_{1}),...,\cos(\alpha_{n})|K_{n})=\\
(-1)^{n/2}\frac{1}{2^{n}\prod_{j=1}^{n}\sin(\alpha_{j})}\sum_{i_{1}%
\in\{-1,1\}}...\sum_{i_{n+k}\in\{-1,1\}}(-1)^{\sum_{l=1}^{n}(i_{l}+1)/2}\\
\frac{\sum_{k=0}^{n-1}(-1)^{k}\sum_{S_{k}\subseteq\mathcal{K}_{n}}^{\prime
}\rho_{S_{k}}\sin(B_{1,\ldots,n}-b_{S_{k}})}{\prod_{j=1}^{n}\prod_{m=j+1}%
^{n}(1-2\rho_{jm}\cos(\beta_{j,m}(i_{j},i_{m}))+\rho_{jm}^{2})}%
\end{gather*}
where $S_{k}$ denotes any subset of $\mathcal{K}_{n}$ that contains $k$
elements and $\sum_{S_{k}\in K_{n}}^{\prime}$ means summation over all $S_{k}$.
\end{theorem}

\begin{proof}
Let us consider (\ref{_t}) first. Keeping in mind assertions of Proposition
\ref{ilocz} we see that $f_{T}(\cos(\alpha_{1}),...,\cos(\alpha_{n})|K_{n})$
is the sum of $2^{n}$ summands depending on different arrangement of values of
variables $i_{k}\in\{-1,1\},$ $k\allowbreak=\allowbreak1,...,n.$ Each summand
is equal to cosine taken at $\sum_{j=1}^{n}i_{j}s_{j}\alpha_{j}.$ Recalling
the definition of numbers $s_{j}$ we see that in such sum $s_{mj},$ $1\leq
m<j\leq n$ appears twice, once as $s_{mj}\alpha_{m}i_{m}$ and secondly as
$s_{mj}\alpha_{j}i_{j}.$ Or in other words, we have $\sum_{j=1}^{n}i_{j}%
s_{j}\alpha_{j}\allowbreak=\allowbreak\sum_{m=1}^{n-1}\sum_{j=m+1}^{n}%
s_{mj}(\alpha_{m}i_{m}+\alpha_{j}i_{j}).$ Having this in mind, we can now
apply summation formula (\ref{ksc}) with $\beta\allowbreak=\allowbreak0$ and
have summed each cosine with a particular system of values of the set
$\{i_{j}:j\allowbreak=\allowbreak1,...,n\}.$ Now it remains to sum over, all
such systems of values.

As far as other assertions are concerned, we use the definition of Chebyshev
polynomials of the second kind, formulae presented in Proposition \ref{ilocz}.
We have in this case $\sum_{j=1}^{n}i_{j}(s_{j}+1)\alpha_{j}\allowbreak
=\allowbreak\sum_{j=1}^{n}i_{j}\alpha_{j}\allowbreak+\allowbreak\sum
_{m=1}^{n-1}\sum_{j=m+1}^{n}s_{mj}(\alpha_{m}i_{m}+\alpha_{j}i_{j}).$ As the
result we deal with signed sum of either sines or cosines depending on the
fact if $n(n-1)/2$ (the number of different $s_{mj},$ $1\leq m<j\leq n$ ) is
odd or even. Now again we refer to either (\ref{kss}) or (\ref{ksc}) depending
on the parity of $n(n-1)/2$ this time with $\beta\allowbreak=\allowbreak
\sum_{j=1}^{n}i_{j}\alpha_{j}.$
\end{proof}

\begin{corollary}
Both functions $f_{T}(\mathbf{x}|K_{n})$ and $f_{U}(\mathbf{x}|K_{n})$ are
rational functions of all its arguments. Moreover, they have the same
denominators given by the following formula:%
\[
V_{n}(\mathbf{x|}K_{n})=\prod_{j=1}^{n-1}\prod_{k=j+1}^{n}w_{2}(x_{j}%
,x_{k}|\rho_{ij}),
\]
where $w_{2}$ is given by the formula (\ref{w2}).
\end{corollary}

\begin{proof}
First of all, notice that following formulae given in Theorem \ref{kibble} the
functions $f_{T}(\mathbf{x}|K_{n})$ and $f_{U}(\mathbf{x}|K_{n})$ are rational
functions of $x_{1}\allowbreak=\allowbreak\cos(\alpha_{1}),...,x_{n}%
\allowbreak=\allowbreak\cos(\alpha_{n}).$ Moreover, it is easy to notice that
all formulae have the same denominators. To find these denominators notice
that the factors in each denominator referring to $(i_{j},i_{m})$ and
$(-i_{j},i_{m})$ are the same since cosine is an even function and that
cosines appear solely in denominators. Further, we can group factors
$(1-2\rho_{jm}\cos(\beta_{j,m}(i_{j},i_{m}))+\rho_{jm}^{2})$ and
$(1-2\rho_{jm}\cos(\beta_{j,m}(i_{j},-i_{m}))+\rho_{jm}^{2})$ and apply
(\ref{pro2})
\begin{gather*}
(1-2\rho_{jm}\cos(\beta_{j,m}(i_{j},i_{m}))+\rho_{jm}^{2})(1-2\rho_{jm}%
\cos(\beta_{j,m}(i_{j},-i_{m}))+\rho_{jm}^{2})\\
=w_{2}(\cos(\alpha_{j}),\cos(\alpha_{m})|\rho_{jm}).
\end{gather*}
since $\beta_{n,m}(i_{n},i_{m})\allowbreak=\allowbreak i_{n}\alpha_{n}%
+i_{m}\alpha_{m}.$
\end{proof}

\begin{corollary}
Let us denote $\beta_{kj}=i_{k}\alpha_{k}+i_{j}\alpha_{j},$ $k\allowbreak
=\allowbreak1,2,$ $j=2,3,$ $k<j,$ $p\allowbreak=\allowbreak\rho_{12}\rho
_{13}\rho_{23},$ $B_{1,2,3}\allowbreak=\sum_{j=1}^{3}i_{j}a_{j},$
\newline$\allowbreak$%
\begin{gather*}
c(i_{1},i_{2},i_{3},\alpha_{1},\alpha_{2},\alpha_{3},\rho_{12},\rho_{13}%
,\rho_{23})\allowbreak=\allowbreak(1-\sum_{1\leq k<j\leq3}\rho_{k,j}\cos
(\beta_{k,j})\allowbreak+\allowbreak\\
p\sum_{1\leq k<j\leq3}\rho_{k,j}^{-1}\cos(2B_{1,2,3}\allowbreak-\allowbreak
\beta_{kj})\allowbreak-\allowbreak p\cos(2B_{1,2,3}))/\prod_{1\leq k<j\leq
3}(1-\rho_{kj}\cos(\beta_{kj})+\rho_{kj}^{2}),
\end{gather*}

\begin{gather*}
s(i_{1},i_{2},i_{3},\alpha_{1},\alpha_{2},\alpha_{3},\rho_{12},\rho_{13}%
,\rho_{23})\allowbreak=(\sin(B_{1,2,3})(1+p)\\
-(\rho_{12}\sin(i_{3}\alpha_{3})+\rho_{13}\sin(i_{2}\alpha_{2})+\rho_{23}%
\sin\left(  i_{1}\alpha_{2}\right)  )\allowbreak\\
-\allowbreak(\rho_{12}\rho_{13}\sin(i_{1}\alpha_{1})+\rho_{12}\rho_{23}%
\sin(i_{2}\alpha_{2})+\rho_{13}\rho_{23}\sin(i_{3}\alpha_{3}))\\
/\allowbreak\prod_{1\leq k<j\leq3}(1-\rho_{kj}\cos(\beta_{kj})+\rho_{kj}^{2}).
\end{gather*}
Then:

i) $f_{T}(\cos(\alpha_{1}),\cos(\alpha_{2}),\cos(\alpha_{3}),\rho_{12}%
,\rho_{13},\rho_{23})=$\newline$\frac{1}{4}\sum_{i_{2}\in\{-1,1\}}\sum
_{i_{3}\in\{-1,1\}}c(1,i_{2},i_{3},\alpha_{1},\alpha_{2},\alpha_{3},\rho
_{12},\rho_{13},\rho_{23}),$

ii) $f_{U}(\cos(\alpha_{1}),\cos(\alpha_{2}),\cos(\alpha_{3}),\rho_{12}%
,\rho_{13},\rho_{23})\allowbreak=\allowbreak$\newline$\frac{1}{8}\sum
_{i_{1}\in\{-1,1\}}\sum_{i_{2}\in\{-1,1\}}\sum_{i_{3}\in\{-1,1\}}%
(-1)^{\sum_{k=1}^{3}(i_{k}+1)/2}s(i_{1},i_{2},i_{3},\alpha_{1},\alpha
_{2},\alpha_{3},\rho_{12},\rho_{13},\rho_{23})$

$p/\rho_{kj}$ in case of $\rho_{kj}\allowbreak=\allowbreak0$ is understood as
the limit when $\rho_{kj}\allowbreak\rightarrow\allowbreak0.$

iii) $f_{U}(x,y,z,\rho_{12},\rho_{13},\rho_{23})\allowbreak=\allowbreak
(4\rho_{12}\rho_{13}(\rho_{23}-\rho_{12}\rho_{13})(1-\rho_{23}^{2}%
)x^{2}\allowbreak+\allowbreak4\rho_{12}\rho_{23}(\rho_{13}-\rho_{12}\rho
_{23})(1-\rho_{13}^{2})y^{2}\allowbreak+\allowbreak4\rho_{13}\rho_{23}%
(\rho_{12}-\rho_{13}\rho_{23})(1-\rho_{12}^{2})z^{2}\allowbreak-\allowbreak
4(\rho_{13}-\rho_{12}\rho_{23})(\rho_{23}-\rho_{12}\rho_{13})(1+\rho_{12}%
\rho_{13}\rho_{23})xy\allowbreak-\allowbreak4(\rho_{12}-\rho_{13}\rho
_{23})(\rho_{23}-\rho_{12}\rho_{13})(1+\rho_{12}\rho_{13}\rho_{23}%
)xz\allowbreak-\allowbreak4(\rho_{13}-\rho_{12}\rho_{23})(\rho_{12}-\rho
_{23}\rho_{13})(1+\rho_{12}\rho_{13}\rho_{23})yz\allowbreak+\allowbreak
(1-\rho_{12}^{2})(1-\rho_{13}^{2})(1-\rho_{23}^{2})(1-\rho_{12}\rho_{13}%
\rho_{23}))$\newline$/(w_{2}(x,y|\rho_{12})w_{2}(x,z|\rho_{13})w_{2}%
(y,z|\rho_{23}))$
\end{corollary}

\begin{proof}
First of all, notice that $\sum_{k=1}^{2}\sum_{j=k+1}^{3}\beta_{kj}%
\allowbreak=\allowbreak2B_{1,2,3}$ hence in particular $B_{1,2,3}%
\allowbreak-\allowbreak\sum_{k=1}^{2}\sum_{j=k+1}^{3}\beta_{kj}\allowbreak
=\allowbreak-B_{1,2,3}.$ Then the formula i) is clear based on (\ref{ksc})
with $\beta\allowbreak=\allowbreak B_{1,2,3}$. To get ii) notice that
$B_{1,2,3}\allowbreak-\allowbreak\beta_{12}\allowbreak=\allowbreak i_{3}%
\alpha_{3}$ and $B_{1,2,3}\allowbreak-\allowbreak\beta_{12}\allowbreak
-\allowbreak\beta_{13}\allowbreak=\allowbreak-i_{1}\alpha_{1},$ similarly for
the other pairs $(1,3)$ and $(2,3)$. Recall also that $B_{1,2,3}%
\allowbreak-\allowbreak\sum_{k=1}^{2}\sum_{j=k+1}^{3}\beta_{kj}\allowbreak
=\allowbreak-B_{1,2,3}.$ Now based on (\ref{kss}) ii) is also clear.

iii) was obtained with the help of Mathematica.
\end{proof}

\begin{remark}
With the help of Mathematica one can show, for example, that the numerator of
$f_{T}(x,y,z|K_{3})$ is a polynomial of degree $6$ and it consists of $265$
monomials. Numerical simulation suggest that it is a nonnegative on
$(-1,1)^{3}.$ Unfortunately $f_{U}(x,y,z|K_{3})$ is not nonnegative there
since we have for example \newline$f_{U}(-.9,-.95,.94,|\left[
\begin{array}
[c]{ccc}%
0 & .6 & .8\\
.6 & 0 & .9\\
.8 & .9 & 0
\end{array}
\right]  )\allowbreak=\allowbreak-0.0912121.$ Besides notice that it happens
in the case when matrix $\left[
\begin{array}
[c]{ccc}%
1, & .6 & .8\\
.6 & 1 & .9\\
.8 & .9 & 1
\end{array}
\right]  $ is positive definite. This observation is in accordance with the
general negative result presented in \cite{SzablKib} Theorem 1. Recall that
\cite{SzablKib} concerns something like generalization of $f_{U}$ to all
parameters $q\in(-1,1)$ taking into account that $q$-Hermite polynomials
$H_{n}(x|q)$ can be identified for $q\allowbreak=\allowbreak0$ with
polynomials $U_{n}(x/2).$ The example presented in \cite{SzablKib} concerns
the case (adopted to $q\allowbreak=\allowbreak0)$ when say $\rho
_{12}\allowbreak=\allowbreak0.$ Hence we see that there are many sets of $6$
tuples $x,y,z,\rho_{12},\rho_{13},\rho_{23}$ leading to negative values of
$f_{U}.$
\end{remark}

\section{Remarks on generalization\label{gen}}

In this section, we are going firstly to present $q-$generalization of the
Chebyshev of the first kind and secondly present some remarks and observations
that might help to obtain formulae similar to the ones presented in Theorem
\ref{main} with Chebyshev polynomials replaced by the so-called $q-$Hermite
$\left\{  h_{n}\right\}  $ and related polynomials. $q$ is here a certain real
(in general) number such that $\left\vert q\right\vert <1.$ Since in the
previous chapters we considered, so to say, the case $q=0$ we will assume in
this chapter that $q\neq0.$

To proceed further we need to recall certain notions used in $q-$series
theory: $\left[  0\right]  _{q}\allowbreak=\allowbreak0;$ $\left[  n\right]
_{q}\allowbreak=\allowbreak1+q+\ldots+q^{n-1},$ $\left[  n\right]
_{q}!\allowbreak=\allowbreak\prod_{j=1}^{n}\left[  j\right]  _{q},$ with
$\left[  0\right]  _{q}!\allowbreak=1,$%
\[%
\genfrac{[}{]}{0pt}{}{n}{k}%
_{q}\allowbreak=\allowbreak\left\{
\begin{array}
[c]{ccc}%
\frac{\left[  n\right]  _{q}!}{\left[  n-k\right]  _{q}!\left[  k\right]
_{q}!} & , & 0\leq k\leq n\\
0 & , & otherwise
\end{array}
\right.  .
\]
$\binom{n}{k}$ will denote ordinary, well known binomial coefficient. \newline
It is useful to use the so-called $q-$Pochhammer symbol for $n\geq1:$%
\[
\left(  a|q\right)  _{n}=\prod_{j=0}^{n-1}\left(  1-aq^{j}\right)  ,~~\left(
a_{1},a_{2},\ldots,a_{k}|q\right)  _{n}\allowbreak=\allowbreak\prod_{j=1}%
^{k}\left(  a_{j}|q\right)  _{n}.
\]
with $\left(  a|q\right)  _{0}\allowbreak=\allowbreak1$. Note that $n$ can be
equal to $\infty,$ then the $q-$Pochhammer symbol is well defined provided
$\left\vert q\right\vert <1.$

Often $\left(  a|q\right)  _{n},$ as well as, $\left(  a_{1},a_{2}%
,\ldots,a_{k}|q\right)  _{n}$ will be abbreviated to $\left(  a\right)  _{n}$
and \newline$\left(  a_{1},a_{2},\ldots,a_{k}\right)  _{n},$ if it will not
cause misunderstanding.

It is easy to notice that $\left(  q\right)  _{n}=\left(  1-q\right)
^{n}\left[  n\right]  _{q}!$ and that%
\[%
\genfrac{[}{]}{0pt}{}{n}{k}%
_{q}\allowbreak=\allowbreak\allowbreak\left\{
\begin{array}
[c]{ccc}%
\frac{\left(  q\right)  _{n}}{\left(  q\right)  _{n-k}\left(  q\right)  _{k}}
& , & n\geq k\geq0\\
0 & , & otherwise
\end{array}
\right.  .
\]
\newline The above mentioned formula is just an example where direct setting
$q\allowbreak=\allowbreak1$ is senseless however, the passage to the limit
$q\longrightarrow1^{-}$ makes sense.

Notice that in particular $\left[  n\right]  _{1}\allowbreak=\allowbreak
n,\left[  n\right]  _{1}!\allowbreak=\allowbreak n!,$ $%
\genfrac{[}{]}{0pt}{}{n}{k}%
_{1}\allowbreak=\allowbreak\binom{n}{k},$ $(a)_{1}\allowbreak=\allowbreak1-a,$
$\left(  a;1\right)  _{n}\allowbreak=\allowbreak\left(  1-a\right)  ^{n}$ and
$\left[  n\right]  _{0}\allowbreak=\allowbreak\left\{
\begin{array}
[c]{ccc}%
1 & if & n\geq1\\
0 & if & n=0
\end{array}
\right.  ,$ $\left[  n\right]  _{0}!\allowbreak=\allowbreak1,$ $%
\genfrac{[}{]}{0pt}{}{n}{k}%
_{0}\allowbreak=\allowbreak1,$ $\left(  a;0\right)  _{n}\allowbreak
=\allowbreak\left\{
\begin{array}
[c]{ccc}%
1 & if & n=0\\
1-a & if & n\geq1
\end{array}
\right.  .$

$i$ will denote, as before, the imaginary unit, unless otherwise clearly
stated. In the sequel we will need also the so-called $q-$Hermite polynomials.
There exists a very large literature on the properties as well as applications
of these polynomials. Let us recall only that the three-term recurrence
satisfied by these polynomials is the following
\[
h_{n+1}(x|q)=2xh_{n}(x|q)-(1-q^{n})h_{n-1}(x|q),
\]
with $h_{-1}(x|q)\allowbreak=\allowbreak0,$ $h_{0}(x|q)\allowbreak
=\allowbreak1.$ It is well known that the density, which makes these
polynomials orthogonal is the following
\[
f_{h}\left(  x|q\right)  =\frac{2\left(  q\right)  _{\infty}\sqrt{1-x^{2}}%
}{\pi}\prod_{k=1}^{\infty}l\left(  x|q^{k}\right)  ,
\]
where $l\left(  x|a\right)  =(1+a)^{2}-4x^{2}a.$ Moreover, generating
functions of these polynomials, are equal to:%
\begin{equation}
\sum_{j=0}^{\infty}\frac{t^{j}}{\left(  q\right)  _{j}}h_{j}\left(
x|q\right)  =\frac{1}{\prod_{k=0}^{\infty}v\left(  x|tq^{k}\right)  },
\label{ch1}%
\end{equation}
where $v\left(  x|a\right)  =1-2ax+a^{2}.$

\begin{remark}
For the sake of completeness of the paper, let us recall that $h_{n}%
(x|0)\allowbreak=\allowbreak U_{n}(x),$ for $n\geq-1.$
\end{remark}

\subsection{Conjectures, remarks and interesting identities\label{Ident}}

Theorem \ref{main} suggests the new method of summing characteristic
functions. One can formulate it in the following way.

\emph{Suppose, that we can guess, that the form of certain multivariate
characteristic function, say for example }%
\begin{equation}
\chi_{n}^{(l_{1},\ldots l_{n})}(x_{1},\ldots,x_{n}|\rho,q)=\sum_{j\geq0}%
\frac{\rho^{j}}{(q)_{j}}\prod_{k=1}^{n}h_{j+l_{k}}(x_{k}|q), \label{gen_ch}%
\end{equation}
\emph{where} \emph{numbers }$l_{1},\ldots,l_{n}$ \emph{are integer and
}$\left\vert \rho\right\vert ,\left\vert q\right\vert <1$, \emph{is of the
form of the ratio of two functions. Moreover, suppose that we can guess the
form of the denominator }$W_{n}(x_{1},\ldots,x_{n}|\rho,q)$ \emph{of this
ratio. Then the numerator can be obtained by the formula similar to
(\ref{diff}) i.e. by: }%
\[
\sum_{j=0}^{\infty}\rho^{j}\sum_{k=0}^{j}\frac{1}{k!}\left.  \frac{d^{k}%
}{d\rho^{k}}W_{n}(x_{1},...,x_{n+m}|\rho,q)\right\vert _{\rho=0}\frac
{1}{(q)_{j-k}}\prod_{s=1}^{n}h_{j-k+l_{s}}(x_{s}|q).
\]

\begin{remark}
There are classes of characteristic functions that have common denominators
like for example bivariate ones described in \cite{Szab5}, Proposition 7 (iv)
or, more generally, bivariate functions of the form similar to (\ref{gen_ch})
that were considered by Carlitz in \cite{Carlitz72}. The point is that all
these functions are at most bivariate. There are no results concerning more
variables. Thus we have the following conjecture.
\end{remark}

\begin{conjecture}
Functions $\chi_{n}^{(l_{1},\ldots l_{n})}(x_{1},\ldots,x_{n}|\rho,q)$ for all
$n,m,l_{1},\ldots,l_{n}$ are the ratios of some functions with the common
denominators of the form%
\[
W_{n}(x_{1},\ldots,x_{n}|\rho,q)=\prod_{i=0}^{\infty}w_{n}(x_{1}%
,...,x_{n}|\rho q^{i}),
\]
where functions $w_{n}(x_{1},\ldots,x_{n}|\rho)$ are given by the iterative
relationship (\ref{rek}).
\end{conjecture}

\subsubsection{One-dimensional case}

Now we will present a one-dimensional example, in order to show that even in
this simplest case we obtain interesting identities. In this example, we will,
so to say, derive once more formula (\ref{ch1}). First of all, notice that
$(1-ae^{i\varphi})(1-ae^{-i\varphi})\allowbreak=\allowbreak1+a^{2}-2ax$
$\overset{df}{=}v(x|a)$ where $x\allowbreak=\allowbreak\cos\varphi$. Moreover,
we have:%
\[
W_{1}(x|\rho,q)=\prod_{j=0}^{\infty}v(x|\rho q^{j})\allowbreak=\allowbreak
(\rho e^{i\varphi})_{\infty}(\rho e^{-i\varphi})_{\infty}.
\]

Let us denote indirectly function $d_{n}(x|q)$ by the relationship: $\frac
{n!}{(q)_{n}}d_{n}(x|q)\allowbreak=\allowbreak\left.  \frac{d^{n}}{d\rho^{n}%
}W_{1}(x|\rho,q)\right\vert _{\rho=0}.$ Notice that $d_{n}(x|q)$ are
coefficients of the expansion of $W_{1}(x|\rho,q)$ in the following series
\begin{equation}
W_{1}(x|\rho,q)=\sum_{n\geq0}\frac{\rho^{n}}{(q)_{n}}d_{n}(x|q). \label{expW}%
\end{equation}
For the sake of symmetry let us also denote by $f_{n}(x|q)$ coefficients of
the expansion $1/W_{1}(x|\rho,q)$ in the following series
\[
1/W_{1}(x|\rho,q)\allowbreak=\allowbreak\sum_{n\geq0}\frac{\rho^{n}}{(q)_{n}%
}f_{n}(x|q).
\]

\begin{remark}
Let us recall polynomials $\left\{  b_{n}\right\}  $ defined in \cite{bms} and
later analyzed in \cite{Szab-rev}(2.43). These polynomials satisfy the
following three term recurrence :%
\[
b_{n+1}(x|q)=-2q^{n}xb_{n}(x|q)+q^{n-1}(1-q^{n})b_{n-1}(x|q),
\]
with $b_{-1}(x|q)\allowbreak=\allowbreak0,$ $b_{1}(x|q)\allowbreak
=\allowbreak1.$ Moreover, as it follows from \cite{SzablAW}(3.18) after some
trivial transformation polynomials $\left\{  b_{n}\right\}  $ satisfy the
following identity:%
\begin{equation}
\sum_{j=1}^{n}%
\genfrac{[}{]}{0pt}{}{n}{j}%
_{q}b_{n-j}(x|q)h_{j+k}(x|q)=\left\{
\begin{array}
[c]{ccc}%
0 & if & k<n\\
(-1)^{n}q^{\binom{n}{2}}\frac{(q)_{k}}{(q)_{k-n}}h_{k-n}(x|q) & if & k\geq n
\end{array}
\right.  . \label{idb}%
\end{equation}

Recall also that the two families of polynomials $\left\{  h_{n}\right\}  $
and $\left\{  b_{n}\right\}  $ are related to one another by
\[
b_{n}(x|q)=(-1)^{n}q^{\binom{n}{2}}h_{n}(x|q^{-1}),
\]
for $q\neq0$ and for $q\allowbreak=\allowbreak0$ we have $b_{-1}%
(x|0)\allowbreak=\allowbreak b_{n}(x|0)\allowbreak=\allowbreak0$ for $n\geq3,$
$b_{1}(x|q)\allowbreak=\allowbreak-2x,$ $b_{2}(x|0)\allowbreak=\allowbreak1.$

In the sequel when considering the case $q\allowbreak=\allowbreak0$ we will
understand as the limit with $q\rightarrow0$ in the function in question.
\end{remark}

One can notice that, we have
\[
\frac{n!}{(q)_{n}}f_{n}(x|q)\allowbreak=\allowbreak\left.  \frac{d^{n}}%
{d\rho^{n}}W_{1}^{-1}(x|\rho,q)\right\vert _{\rho=0}.
\]
We have the following lemma.

\begin{lemma}
\label{1-dim}For $\left\vert x\right\vert \leq1,\left\vert q\right\vert <1,$
we have
\begin{align}
d_{n}(x|q)\allowbreak &  =\allowbreak b_{n}(x|q),\label{1dimb}\\
f_{n}(x|q)  &  =h_{n}(x|q). \label{1dimh}%
\end{align}

\end{lemma}

\begin{proof}
To prove (\ref{1dimb}) let us recall formula (1.7) of \cite{bms}.
\[
W_{1}(x|\rho,q)=\sum_{j\geq0}\frac{\rho^{j}}{(q)_{j}}b_{j}(x|q).
\]
To get (\ref{1dimh}) we recall (\ref{ch1}). The separate proof is needed for
the case $q\allowbreak=\allowbreak0.$ Then $W_{1}(x,\rho,0)\allowbreak
=\allowbreak v(x|\rho)\allowbreak=\allowbreak1-2x\rho+\rho^{2}$ which
confronted with our definition of polynomials $b_{n}$ for $q\allowbreak
=\allowbreak0$ shows that the (\ref{1dimb}) is true for this case also.
\end{proof}

Now we see that following, adapted to the present situation, formula
(\ref{diff}) we have, for $\left\vert q\right\vert ,\left\vert \rho\right\vert
<1 $ and $\left\vert x\right\vert \leq1$.
\begin{align*}
\chi_{1}^{t}(x|\rho,q)\allowbreak &  =\allowbreak\sum_{j=0}^{\infty}\frac
{\rho^{j}}{(q)_{j}}h_{t+j}(x|q)\allowbreak=\allowbreak\frac{1}{W_{1}%
(x|\rho,q)}\\
&  \times\sum_{j=0}^{\infty}\rho^{j}\sum_{m=0}^{j}\frac{1}{(j-m)!}%
\frac{(j-m)!}{(q)_{j-m}(q)_{m}}b_{j-m}(x|q)h_{m+t}(x|q)\\
&  =\allowbreak\frac{1}{W_{1}(x|\rho,q)}\sum_{j=0}^{\infty}\frac{\rho^{j}%
}{(q)_{j}}\sum_{m=0}^{j}%
\genfrac{[}{]}{0pt}{}{j}{m}%
_{q}b_{j-m}(x|q)h_{m+t}(x|q)\\
&  =\frac{1}{W_{1}(x|\rho,q)}\sum_{j=0}^{t}%
\genfrac{[}{]}{0pt}{}{t}{j}%
_{q}(-\rho)^{j}q^{\binom{j}{2}}h_{t-j}(x|q).
\end{align*}
In particular, for $t\allowbreak=\allowbreak0,$ we get once more formula
(\ref{ch1}). This can be regarded as yet another prove of this formula since
we started from (\ref{expW}).

\subsubsection{Two-dimensional case}

Again, as before, let us denote\newline$\frac{n!}{(q)_{n}}d_{n}^{(2)}%
(x,y|q)\allowbreak=\allowbreak\left.  \frac{d^{n}}{d\rho^{n}}W_{2}%
(x,y|\rho,q)\right\vert _{\rho=0},$ $\frac{n!}{(q)_{n}}f_{n}^{(2)}%
(x,y|q)\allowbreak=\allowbreak\left.  \frac{d^{n}}{d\rho^{n}}W_{2}%
^{-1}(x,y|\rho,q)\right\vert _{\rho=0}$, where $W_{2}(x,y|\rho,q)\allowbreak
=\allowbreak\prod_{j=0}^{\infty}w_{2}(x,y|\rho q^{j}),$ with $w_{2}(x,y|a)$
defined by (\ref{w2}).

\begin{lemma}
\label{2-dim}For $\theta,\varphi\in\lbrack0,2\pi),\left\vert q\right\vert <1,$
we have
\begin{align}
d_{n}^{(2)}(\cos\theta,\cos\varphi|q)\allowbreak &  =\allowbreak\sum_{m=0}^{n}%
\genfrac{[}{]}{0pt}{}{n}{m}%
_{q}b_{m}(\cos(\theta+\varphi)|q)b_{n-m}(\cos(\theta-\varphi)|q),\label{d2b}\\
f_{n}^{(2)}(\cos\theta,\cos\varphi|q)  &  =\sum_{m=0}^{n}%
\genfrac{[}{]}{0pt}{}{n}{m}%
_{q}h_{m}(\cos(\theta+\varphi)|q)h_{n-m}(\cos(\theta-\varphi)|q). \label{d2h}%
\end{align}

\end{lemma}

\begin{proof}
First of all, notice that $w_{2}(\cos\theta,\cos\varphi|\rho)$ can be
decomposed as
\begin{equation}
w_{2}(\cos\theta,\cos\varphi|\rho)\allowbreak=\allowbreak w_{1}(\cos
(\theta+\varphi)|\rho)(w_{1}(\theta-\varphi)|\rho) \label{w22}%
\end{equation}
hence, taking into account Leibniz rule, we get:
\begin{gather*}
d_{n}^{(2)}(x,y|q)=\frac{(q)_{n}}{n!}\left.  \frac{d^{n}}{d\rho^{n}}%
(W_{1}(\cos(\theta+\varphi)|\rho,q)W_{1}(\cos(\theta-\varphi)|\rho
,q))\right\vert _{\rho=0}\\
=\frac{(q)_{n}}{n!}\sum_{m=0}^{n}\binom{n}{m}\left.  \frac{d^{m}}{d\rho^{m}%
}(W_{1}(\cos(\theta+\varphi)|\rho,q)\right\vert _{\rho=0}\left.  \frac
{d^{n-m}}{d\rho^{n-m}}(W_{1}(\cos(\theta-\varphi)|\rho,q)\right\vert _{r=0}\\
=\frac{(q)_{n}}{n!}\sum_{m=0}^{n}\binom{n}{m}\frac{m!}{(q)_{m}}b_{m}%
(\cos(\theta+\varphi)|q)\frac{(n-m)!}{(q)_{n-m}}b_{n-m}(\cos(\theta
-\varphi)|q).
\end{gather*}
To get (\ref{d2h}), we argue in a similar way using Lemma \ref{1-dim} on the way.
\end{proof}

\begin{theorem}
\label{wazny}We have for $\left\vert x\right\vert ,\left\vert y\right\vert
,\left\vert q\right\vert \in\mathbb{R}$ and all $n\geq0:$%
\begin{gather}
d_{n}^{(2)}(x,y|q)=\label{con1}\\
(-1)^{n}\sum_{j=0}^{\left\lfloor n/2\right\rfloor }(-1)^{j}q^{-\binom{n-2j}%
{2}-j+\binom{j}{2}}\frac{(q)_{n}}{(q)_{j}(q)_{n-2j}}b_{n-2j}(x|q)b_{n-2j}%
(y|q),\nonumber\\
f_{n}^{(2)}(x,y|q)\allowbreak=\allowbreak\sum_{j=0}^{\left\lfloor
n/2\right\rfloor }\frac{(q)_{n}}{(q)_{j}(q)_{n-2j}}h_{n-2j}(x|q)h_{n-2j}(y|q).
\label{con2}%
\end{gather}

\end{theorem}

\begin{proof}
Is shifted to Section \ref{dow}.
\end{proof}

\begin{remark}
Notice that, in accordance with our agreement that the case $q\allowbreak
=\allowbreak0$ will be understood as the limit when $q\rightarrow0,$ we have
$d_{0}^{(2)}(x,y|0)\allowbreak=\allowbreak1,$ $d_{1}^{(2)}(x,y|0)\allowbreak
=\allowbreak-4xy$, $d_{2}^{(2)}(x,y|0)\allowbreak=\allowbreak4(x^{2}%
+y^{2})-2,$ $d_{3}^{(2)}(x,y|0)\allowbreak=\allowbreak-4xy,$ $d_{4}%
^{(2)}(x,y|0)\allowbreak=\allowbreak1$, $d_{n}^{(2)}(x,y|0)\allowbreak
=\allowbreak0$ for all $n\geq4$.
\end{remark}

As a corollary we get the following interesting nontrivial identity involving
polynomials $\left\{  b_{n}\right\}  $ and $\left\{  h_{n}\right\}  .$

\begin{corollary}
For all complex $x,y,q,$ $k\geq0$ and $t,s\in\mathbb{N\cup}\{0\},$we get%
\begin{equation}
\sum_{m=0}^{k}%
\genfrac{[}{]}{0pt}{}{k}{m}%
_{q}d_{m}^{(2)}(x,y|q)h_{k-m+t}(x|q)h_{k-m+s}(y|q)=P_{t,s}^{(k)}(x,y|q)
\label{sumk}%
\end{equation}
where $P_{t,s}^{(k)}(x,y|q)$ is a polynomial of order $t+s$ in $x$ and $y.$

In particular, we have%
\begin{equation}
\sum_{m=0}^{k}%
\genfrac{[}{]}{0pt}{}{k}{m}%
_{q}d_{m}^{(2)}(x,y|q)h_{k-m}(x|q)h_{k-m}(y|q)\allowbreak=\left\{
\begin{array}
[c]{ccc}%
0 & if & k\text{ is odd}\\
(-1)^{l}q^{\binom{l}{2}}(q^{l+1})_{l} & if & k=2l
\end{array}
\right.  . \label{00k}%
\end{equation}

\end{corollary}

\begin{proof}
Knowing that
\[
\sum_{j=0}^{\infty}\frac{\rho^{j}}{(q)_{j}}h_{j+t}(x|q)h_{j+s}(y|q)\allowbreak
=\allowbreak\frac{(\rho^{2})_{\infty}V_{t,s}(x,y|\rho,q)}{W_{2}(x,y|\rho,q)},
\]
for $t,s\in\mathbb{N\cup}\{0\}$, which is a modification of the formula given
in assertion i) of Lemma 3 in \cite{SzablAW}, where $V_{t,s}(x,y|\rho,q)$
denotes certain polynomial of the degree $t+s$ in $x$ and $y$, our expansion
of $W_{2}(x,y|\rho,q)$ and then applying Cauchy multiplication of series get
the identity
\begin{equation}
\sum_{j=0}^{\infty}\frac{\rho^{j}}{(q)_{j}}\sum_{m=0}^{j}%
\genfrac{[}{]}{0pt}{}{j}{m}%
d_{m}^{(2)}(x,y|q)h_{j-m+t}(x|q)h_{j-m+s}(y|q)\allowbreak=\allowbreak
V_{t,s}(x,y|\rho,q)(\rho^{2})_{\infty}, \label{id2wym}%
\end{equation}
true for all $\left\vert x\right\vert ,\left\vert y\right\vert \leq1,$
$\left\vert \rho\right\vert ,\left\vert q\right\vert <1.$ Now knowing the form
of the polynomial $V_{t,s}$ given either in \cite{SzablAW}, \cite{SzabP-M} or
\cite{Szab-rev}, we deduce that the expansion of the polynomial $V_{t,s}$ in
the power series of $\rho$ is of a form of the sum of infinite power series
only in $\rho$ times polynomials of $x$ and $y$ of order at most $t+s$. Hence
it is of the form of the power series in $\rho$ with coefficients being
polynomials in $x$ and $y$ of order at most $t+s.$ Since the linear
combination of polynomials of order $t+s$ is a polynomial of order $t+s.$ A
similar argument can be applied to the product $V_{t,s}(x,y|\rho,q)(\rho
^{2})_{\infty}.$ Now comparing the coefficients of the powers of $\rho$ on the
two sides of (\ref{id2wym}), one proves the first part of the statement.

Now knowing that $V_{0,0}\allowbreak=\allowbreak1$, expanding $\left(
\rho^{2}\right)  _{\infty}$ in a standard way and finally comparing
coefficients by equal powers of $\rho$ we arrive to (\ref{sumk}).
\end{proof}

\section{Proofs\label{dow}}

\begin{proof}
[Proof of Proposition \ref{ilocz}.]We will be using well known formulae for
the product of sines and cosines. The proof is by induction. For
$n\allowbreak=\allowbreak1$ and $k\allowbreak=\allowbreak1$ we have in case of
(\ref{cos}) and $k\allowbreak=\allowbreak0$ $\cos(\alpha)\allowbreak
=\allowbreak\frac{1}{2}(\cos(\alpha)+\cos(-\alpha))$ while in case of
(\ref{sin}) we get%
\begin{gather*}
\sin(\alpha_{1})\cos(\alpha_{2})=\frac{-1}{4}(\sin(-\alpha_{1}-\alpha
_{2})\allowbreak+\allowbreak\sin(-\alpha_{1}+\alpha_{2})-\sin(\alpha
_{1}-\alpha_{2})-\sin\left(  \alpha_{1}+\alpha_{2})\right) \\
=\frac{1}{2}(\sin(\alpha_{1}+\alpha_{2})+\sin(\alpha_{1}-\alpha_{2})).
\end{gather*}
Hence, let us assume that they are true for $n\allowbreak=\allowbreak m.$

In the case of the first one, we have
\begin{gather*}
\prod_{j=1}^{m+1}\cos(\xi_{j})\allowbreak=\allowbreak\cos(\xi_{m+1}%
)\prod_{j=1}^{m}\cos(\xi_{j})\allowbreak=\allowbreak\\
\frac{1}{2^{m}}\sum_{i_{1}\in\{-1,1\}}...\sum_{i_{m}\in\{-1,1\}}\cos
(\sum_{k=1}^{m}i_{k}\xi_{k})\cos(\xi_{m+1})\allowbreak\\
=\allowbreak\newline\frac{1}{2^{m+1}}\allowbreak\times\allowbreak\sum
_{i_{1}\in\{-1,1\}}...\sum_{i_{m}\in\{-1,1\}}(\cos(\sum_{k=1}^{m}i_{k}\xi
_{k}+\xi_{m+1})\allowbreak+\allowbreak\cos(\sum_{k=1}^{m}i_{k}\xi_{k}%
-\xi_{m+1})).
\end{gather*}

Along the way we used the fact that $\cos(\alpha)\cos(\beta)\allowbreak
=\allowbreak(\cos(\alpha-\beta)+\cos(\alpha+\beta))/2.$ Let us also observe
that the product $\prod_{j=1}^{m}\cos(\xi_{j})$ is a sum of cosines of a
certain linear combination of arguments $\xi_{j},$ $j\allowbreak
=\allowbreak1,\ldots,m$ multiplied by $2^{m-1}.$

In the case of the second one we first consider the case of $k\allowbreak
=\allowbreak0$. Assuming that $m$ is even we get:
\begin{gather*}
\prod_{j=1}^{m+1}\sin(\xi_{j})=\allowbreak\sin(\xi_{m+1})\prod_{j=1}^{m}%
\sin(\xi_{j})\allowbreak=\allowbreak(-1)^{m/2}\frac{1}{2^{m}}\allowbreak
\times\allowbreak\\
\sum_{i_{1}\in\{-1,1\}}...\sum_{i_{m}\in\{-1,1\}}(-1)^{\sum_{k=1}^{m}%
(i_{k}+1)/2}\cos(\sum_{k=1}^{m}i_{k}\xi_{k})\sin(\xi_{m+1})\allowbreak\\
=(-1)^{m/2}\frac{1}{2^{m+1}}\allowbreak\sum_{i_{1}\in\{-1,1\}}...\sum
_{i_{m}\in\{-1,1\}}(-1)^{\sum_{k=1}^{m}(i_{k}+1)/2}\allowbreak\times\\
(\sin(\sum_{k=1}^{m}i_{k}\xi_{k}\allowbreak+\allowbreak\xi_{m+1}%
)\allowbreak-\allowbreak\sin(\sum_{k=1}^{m}i_{k}\xi_{k}\allowbreak
-\allowbreak\xi_{m+1}))\allowbreak=\\
-(-1)^{m/2}\frac{1}{2^{m+1}}\sum_{i_{m+1}\in\{-1\}}\sum_{i_{1}\in
\{-1,1\}}...\sum_{i_{m}\in\{-1,1\}}(-1)^{\sum_{k=1}^{m+1}(i_{k}+1)/2}\sin
(\sum_{k=1}^{m+1}i_{k}\xi_{k})\allowbreak\\
-(-1)^{m/2}\frac{1}{2^{m+1}}\sum_{i_{m+1}\in\{-1\}}\sum_{i_{1}\in
\{-1,1\}}...\sum_{i_{m}\in\{-1,1\}}(-1)^{\sum_{k=1}^{m+1}(i_{k}+1)/2}\sin
(\sum_{k=1}^{m+1}i_{k}\xi_{k}).
\end{gather*}
$\allowbreak$\newline We used the fact that $\sin(\alpha)\cos(\beta
)\allowbreak=\allowbreak(\sin(\alpha-\beta)+\sin(\alpha+\beta))/2.$ The case
of $m$ odd is treated in the similar way.

Now to consider general case we expand both products of sines and cosines.
\end{proof}

\begin{proof}
[Proof of Lemma \ref{aux1}](\ref{ksc}) Using the Euler's identity $\cos
(\theta)\allowbreak=\allowbreak(e^{i\theta}\allowbreak+\allowbreak
e^{-i\theta})/2$ we get
\[
\cos(\beta+\sum_{j=1}^{n}k_{j}\alpha_{j})\allowbreak=\allowbreak\exp
(i\beta+\sum_{j=1}^{n}ik_{j}\alpha_{j})/2\allowbreak+\allowbreak\exp
(-i\beta-\sum_{j=1}^{n}ik_{j}\alpha_{j})/2.
\]
So
\[
\sum_{k_{1}\geq0}...\sum_{k_{n}\geq0}(\prod_{j=1}^{n}\rho_{j}^{k_{j}}%
)\exp(i\beta+\sum_{j=1}^{n}ik_{j}\alpha_{j})/2\allowbreak=\allowbreak\frac
{1}{2}\exp(i\beta)\prod_{j=1}^{n}\frac{1}{1-\rho_{j}\exp(i\alpha_{j})}.
\]
Similarly:
\[
\sum_{k_{1}\geq0}...\sum_{k_{n}\geq0}(\prod_{j=1}^{n}\rho_{j}^{k_{j}}%
)\exp(-i\beta-\sum_{j=1}^{n}ik_{j}\alpha_{j})/2\allowbreak=\allowbreak\frac
{1}{2}\exp(-i\beta)\prod_{j=1}^{n}\frac{1}{1-\rho_{j}\exp(-i\alpha_{j})}.
\]
Thus%
\begin{gather*}
\sum_{k_{1}\geq0}...\sum_{k_{n}\geq0}(\prod_{j=1}^{n}\rho_{j}^{k_{j}}%
)\cos(\beta+\sum_{j=1}^{n}ik_{j}\alpha_{j})\\
=\frac{\exp(i\beta)\prod_{j=1}^{n}(1-\rho_{j}\exp(-i\alpha_{j}))+\exp
(-i\beta)\prod_{j=1}^{n}(1-\rho_{j}\exp(i\alpha_{j}))}{2\prod_{j=1}^{n}%
(1+\rho_{j}^{2}-2\rho_{j}\cos(\alpha_{j}))}.
\end{gather*}
Now, notice that
\begin{gather*}
\exp(-i\beta)\prod_{j=1}^{n}(1-\rho_{j}\exp(i\alpha_{j}))\allowbreak=\\
\allowbreak\sum_{j=1}^{n}(-1)^{j}\sum_{M_{j,n}\subseteq S_{n}}\prod_{k\in
M_{j,n}}\rho_{k}\exp(-i\beta+i\sum_{k\in M_{j,n}}\alpha_{k}).
\end{gather*}

To verify (\ref{kss}), we use the fact that $\sin(\theta)\allowbreak
=\allowbreak(e^{i\theta}\allowbreak-\allowbreak e^{-i\theta})/2$ getting%
\[
\sin(\beta+\sum_{j=1}^{n}k_{j}\alpha_{j})\allowbreak=\allowbreak\exp
(i\beta+\sum_{j=1}^{n}ik_{j}\alpha_{j})/2i\allowbreak-\newline\allowbreak
\exp(-i\beta-\sum_{j=1}^{n}ik_{j}\alpha_{j})/2i.
\]
So we have:
\[
\sum_{k_{1}\geq0}...\sum_{k_{n}\geq0}(\prod_{j=1}^{n}\rho_{j}^{k_{i}}%
)\exp(i\beta+i\sum_{j=1}^{n}k_{j}\alpha_{j})/2i\allowbreak=\allowbreak
\exp(i\beta)\frac{1}{2i}\prod_{j=1}^{n}\frac{1}{1-\rho_{j}\exp(i\alpha_{j})}.
\]
Similarly we get\newline%
\[
\sum_{k_{1}\geq0}...\sum_{k_{n}\geq0}(\prod_{j=1}^{n}\rho_{j}^{k_{j}}%
)\exp(-i\beta-i\sum_{j=1}^{n}k_{j}\alpha_{j})/2i\allowbreak=\allowbreak
\exp(-i\beta)\frac{1}{2i}\prod_{j=1}^{n}\frac{1}{1-\rho_{j}\exp(-i\alpha_{j}%
)}.
\]
So
\begin{gather*}
\sum_{k_{1}\geq0}...\sum_{k_{n}\geq0}(\prod_{j=1}^{n}\rho_{j}^{k_{j}}%
)\sin(\beta+\sum_{j=1}^{n}k_{j}\alpha_{j})\allowbreak=\allowbreak\\
\frac{1}{2i}\frac{\exp(i\beta)\prod_{j=1}^{n}(1-\rho_{j}\exp(-i\alpha
_{j}))-\exp(-i\beta)\prod_{j=1}^{n}(1-\rho_{j}\exp(i\alpha_{j}))}{\prod
_{j=1}^{n}(1+\rho_{j}^{2}-2\rho_{j}\cos(\alpha_{j}))}\allowbreak.
\end{gather*}

\end{proof}

\begin{proof}
[Proof of Theorem \ref{main}]The proof is based on the following observation.
First one is that we convert products Chebyshev polynomials to the products of
$\sin(j\alpha_{s}+(t_{s}+1)\alpha_{s})$ and $\cos(j\alpha_{s}+t_{s}\alpha
_{s})$ according to (\ref{Czebysz}). Secondly we change these products to sums
of either cosines if $n$ is even or zero or sines if $n$ is odd according to
the assertion of the Proposition \ref{ilocz}. The arguments of these sines and
cosines are the linear combinations of the arguments of sines and cosines that
were participating in the products. The coefficients of these linear
combinations are $j\geq0$ and $i_{m}\in\left\{  -1,1\right\}  ,$
$m\allowbreak=\allowbreak1,\ldots,n+k.$ Thus we can sum first with respect to
$j$ and apply Corollary \ref{suma}. There the r\^{o}le of $\alpha$ plays now
$\sum_{s=1}^{k+n}$ $i_{s}\alpha_{s}$ for chosen combination of $i^{\prime}s$
while the r\^{o}le of $\beta$ similar combination $\sum_{s=1}^{n}i_{s}%
(t_{s}+1)\alpha_{s}+\allowbreak+\allowbreak\sum_{s=n+1}^{n+k}i_{s}t_{s}%
\alpha_{s}.$ The point is that the sum of such sines or cosines with respect
to $j,$ is a ratio of two trigonometric expressions. Moreover all these the
expressions in the denominators depend only on $\sum_{s=1}^{k+n}$ $i_{s}%
\alpha_{s},$ i.e. do not depend on indeces $t_{s}$ (note that denominators of
sums in Corollary \ref{suma} do not depend on $\beta$). For $\alpha_{s}%
\in\mathbb{R}$, $t_{s}\in\mathbb{Z}$, $s=1,...,n+k$, $\left\vert
\rho\right\vert <1$ we have, depending on the parity of $n$, the following equations.

If $n$ is odd then,%
\begin{gather}
\sum_{j\geq0}\rho^{j}\prod_{s=1}^{n}U_{j+t_{s}}(\cos(\alpha_{s}))\prod
_{s=n+1}^{n+k}T_{j+t_{s}}(\cos(\alpha_{s}))=\label{si}\\
\frac{(-1)^{(n+1)/2}}{2^{n+k}\prod_{i=1}^{n}\sin(\alpha_{i})}\sum_{i_{1}%
\in\{-1,1\}}...\sum_{i_{n+k}\in\{-1,1\}}(-1)^{\sum_{k=1}^{n}(i_{k}+1)/2}%
\times\nonumber\\
\frac{(\sin(\sum_{s=1}^{n}i_{s}(t_{s}+1)\alpha_{s}+\sum_{s=n+1}^{n+k}%
i_{s}t_{s}\alpha_{s})-\rho\sin(\sum_{s=1}^{n}i_{s}t_{s}\alpha_{s}+\sum
_{s=n+1}^{n+k}i_{s}(t_{s}-1)\alpha_{s}))}{(1-2\rho\cos(\sum_{s=1}^{n+k}%
i_{s}\alpha_{s})+\rho^{2})},\nonumber
\end{gather}
while, when $n$ is even or zero, we get:$\allowbreak$\newline%
\begin{gather}
\sum_{j\geq0}\rho^{j}\prod_{s=1}^{n}U_{j+t_{s}}(\cos(\alpha_{s}))\prod
_{s=n+1}^{n+k}T_{j+t_{s}}(\cos(\alpha_{s}))=\label{chi}\\
\frac{(-1)^{n/2}}{2^{n+k}\prod_{i=1}^{n}\sin(\alpha_{i})}\sum_{i_{1}%
\in\{-1,1\}}...\sum_{i_{n+k}\in\{-1,1\}}(-1)^{\sum_{k=1}^{n}(i_{k}+1)/2}%
\times\nonumber\\
\frac{\cos(\sum_{s=1}^{n}i_{s}(t_{s}+1)\alpha_{s}+\sum_{s=n+1}^{n+k}i_{s}%
t_{s}\alpha_{s})-\rho\cos(\sum_{s=1}^{n}i_{s}t_{s}\alpha_{s}+\sum
_{s=n+1}^{n+k}i_{s}(t_{s}-1)\alpha_{s})}{(1-2\rho\cos(\sum_{s=1}^{n+k}%
i_{s}\alpha_{s})+\rho^{2})}.\nonumber
\end{gather}

To justify it, we use (\ref{Czebysz}) first, then based on Proposition
\ref{ilocz}, we convert products to sums of sines or cosines (if $n$ is odd
sines if $n$ is even cosines) that are of the following arguments:
\begin{gather*}
\sum_{s=1}^{n}l_{s}((j+1)\alpha_{s}+t_{s}\alpha_{s})\allowbreak+\allowbreak
\sum_{s=n+1}^{n+k}l_{s}(j\alpha_{s}+t_{s}\alpha_{s})\allowbreak\\
=\allowbreak j\sum_{s=1}^{n+k}l_{s}\alpha_{s}+\sum_{s=1}^{n}l_{s}%
(t_{s}+1)\alpha_{s}+\sum_{s=n+1}^{n+k}l_{s}t_{s}\alpha_{s}.
\end{gather*}
Then, we change the order of summation and we sum over $j$ first. We identify
"$\alpha$" with $\sum_{s=1}^{n+k}l_{s}\alpha_{s}$ and "$\beta$" with
$\sum_{s=1}^{n}l_{s}(t_{s}+1)\alpha_{s}\allowbreak+\allowbreak\sum
_{s=n+1}^{n+k}l_{s}t_{s}\alpha_{s}$ and apply formulae (\ref{s_si} or
\ref{s_g_c}) depending on on the case of parity of $n$.

Now let us analyze polynomial $w_{n}.$ Notice that denominator in both
(\ref{si}) and (\ref{chi}) is of the form
\begin{gather}
w_{k+n}(\cos(\alpha_{1}),...,\cos(\alpha_{k+n})|\rho)=\label{wkn}\\
\prod_{i_{1}\in\{-1,1\}}...\prod_{i_{k+n}\in\{-1,1\}}(1-2\rho\cos(\sum
_{s=1}^{n+k}i_{s}\alpha_{s})+\rho^{2}).\nonumber
\end{gather}
To get (\ref{wkn}) we will argue by induction. Let us replace $n+k$ by $m$ to
avoid confusion. To start with $m\allowbreak=\allowbreak1$ for $m\allowbreak
=\allowbreak2$ we recall (\ref{pro2}). Hence (\ref{wkn}) is true for
$m\allowbreak=\allowbreak1,2.$

Let us assume that the formula is true for $m\allowbreak=\allowbreak k+1.$
Hence, taking $\alpha\allowbreak=\allowbreak\alpha_{k+1}$ and $\beta
\allowbreak=\allowbreak\sum_{s=1}^{k}i_{s}\alpha_{s}$ and noting that
$i_{k}^{2}\allowbreak=\allowbreak1$we get:%
\begin{gather*}
w_{k+1}(\cos(\alpha_{1}),...,\cos(\alpha_{k+1})|\rho)=\\
\prod_{i_{2}\in\{-1,1\}}...\prod_{i_{k}\in\{-1,1\}}((1-2\rho\cos(\sum
_{s=1}^{k-1}i_{s}\alpha_{s}+i_{k}(\alpha_{k}-i_{k}\alpha_{k+1}))+\rho^{2})\\
\times(1-2\rho\cos(\sum_{s=1}^{k-1}i_{s}\alpha_{s}+i_{k}(\alpha_{k}%
+i_{k}\alpha_{k+1}))+\rho^{2}))\\
=w_{k}(\cos(\alpha_{1}),...,\cos(\alpha_{k}+\alpha_{k+1})|\rho)w_{k}%
(\cos(\alpha_{1}),...,\cos(\alpha_{k}-\alpha_{k+1})|\rho).
\end{gather*}
by induction assumption. Now it is elementary to see that polynomials $w_{n}$
satisfy relationship (\ref{rek}). Similarly, the remarks concerning degree of
symmetry and the degree of polynomials $w_{n}$ follow directly (\ref{wkn}).

Now, let us multiply both sides of (\ref{si}) and (\ref{chi}) by
$w_{n+k}(x_{1},...,x_{n+k}|\rho)$. We see that this product is equal to the
right hand sides of these equalities with an obvious replacement $\cos
(\alpha_{s})->x_{s},$ $s=1,...,n+k.$ Inspecting (\ref{si}) and (\ref{chi}), we
notice that these right hand sides are polynomials of degree $2(2^{n+k-1}%
-1)+1\allowbreak=\allowbreak2^{n+k}-1$ in $\rho.$ Thus, these polynomials can
be regained by using well known formula:
\[
p_{n}(x)\allowbreak=\allowbreak\sum_{i=0}^{n}x^{n}a_{n}\allowbreak
=\allowbreak\sum_{j=0}^{n}\frac{x^{j}}{j!}\left.  \frac{d^{j}}{dx^{j}}%
p_{n}(x)\right\vert _{x=0}.
\]
This leads directly to the differentiation of the products of $w_{n+k}%
(x_{1},...,x_{n+k}|\rho)$ and right hand side of (\ref{_ktnu}). Now we apply
the Leibniz formula:%
\[
\left.  \frac{d^{n}}{dx^{n}}[f(x)g(x)]\right\vert _{x=0}=\sum_{j=0}^{n}%
\binom{n}{i}\left.  \frac{d^{j}}{dx^{j}}f(x)\right\vert _{x=0}\left.
\frac{d^{n-j}}{dx^{n-j}}g(x)\right\vert _{x=0}.
\]
and notice that
\[
\left.  \frac{d^{k}}{d\rho^{k}}\sum_{j\geq0}\rho^{j}\prod_{s=1}^{n}T_{j+t_{s}%
}(x_{s})\prod_{s=1+n}^{n+k}U_{j+t_{s}}(x_{s})\right\vert _{\rho=0}%
=k!\prod_{s=1}^{n}T_{k+t_{s}}(x_{s})\prod_{s=1+n}^{n+k}U_{k+t_{s}}(x_{s}).
\]
Having this we get directly (\ref{formula}).
\end{proof}

\begin{proof}
[ Proof of the Theorem \ref{wazny}]The proof consists of several steps. First,
we prove that for all $\theta,\varphi\in\mathbb{R}$ we have%
\begin{equation}
\sum_{m=0}^{n}%
\genfrac{[}{]}{0pt}{}{n}{m}%
_{q}h_{m}(\cos(\theta+\varphi)|q)h_{n-m}(\cos(\theta-\varphi)|q)=\sum
_{j=0}^{\left\lfloor n/2\right\rfloor }\frac{(q)_{n}}{(q)_{j}(q)_{n-2j}%
}h_{n-2j}(\cos\theta|q)h_{n-2j}(\cos\varphi|q). \label{exKM}%
\end{equation}
This formula follows, firstly from the fact that we have
\[
\left.  \frac{d^{n}}{d\rho^{n}}W_{1}^{-1}(x|\rho,q)\right\vert _{\rho=0}%
=\frac{n!}{(q)_{n}}h_{n}(x|q),
\]
which follows directly from (\ref{ch1}). Secondly, arguing in the similar way
as in the proof of Lemma \ref{2-dim} we deduce that
\begin{align*}
&  \left.  \frac{d^{n}}{d\rho^{n}}W_{1}^{-1}(\cos(\theta+\varphi)|\rho
,q)W_{1}^{-1}(\cos(\theta-\varphi)|\rho,q)\right\vert _{\rho=0}\\
&  =\frac{n!}{(q)_{n}}\sum_{m=0}^{n}%
\genfrac{[}{]}{0pt}{}{n}{m}%
_{q}h_{m}(\cos(\theta+\varphi)|q)h_{n-m}(\cos(\theta-\varphi)|q).
\end{align*}
Thirdly, we notice that
\[
\frac{1}{W_{1}(\cos(\theta+\varphi)|\rho,q)W_{1}(\cos(\theta-\varphi)|\rho
,q)}\allowbreak=\allowbreak\frac{1}{W_{2}(\cos(\theta),\cos(\varphi)|\rho
,q)},
\]
which follows directly from (\ref{w22}).

Now, let us calculate
\[
\sum_{n\geq0}\frac{\rho^{n}}{(q)_{n}}\sum_{j=0}^{\left\lfloor n/2\right\rfloor
}\frac{(q)_{n}}{(q)_{j}(q)_{n-2j}}h_{n-2j}(\cos(\theta)|q)h_{n-2j}%
(\cos(\varphi)|q).
\]
After changing the order of summation, we get%
\[
\sum_{j\geq0}\frac{\rho^{2j}}{(q)_{j}}\sum_{n\geq2j}\frac{\rho^{n-2j}%
}{(q)_{n-2j}}h_{n-2j}(\cos(\theta)|q)h_{n-2j}(\cos(\varphi)|q)\allowbreak
=\allowbreak\frac{1}{(\rho^{2})_{\infty}}\frac{(\rho^{2})_{\infty}}{W_{2}%
(\cos(\theta),\cos(\varphi)|\rho,q)},
\]
by the binomial and Poisson-Mehler summation theorems. Thus we have proved
(\ref{exKM}) as well as (\ref{con2}) at least for $\left\vert q\right\vert
<1.$ The formula can be easily extended to all values of $q\neq1$ since both
sides are polynomials in $q$. Similarly, we can extend it to all values of $x$
and $y$ by substitution $\cos(\theta)$ by $x$ and $\cos(\varphi)$ by $y.$ Now,
having proven (\ref{exKM}) we recall the definition of polynomials
$b_{n}(x|q)$ given in Lemma \ref{1-dim}, above. Recall also that
\[
(\frac{1}{q}|\frac{1}{q})_{n}=(-1)^{n}q^{-\binom{n+1}{2}}(q)_{n},
\]
and consequently that we have:
\[%
\genfrac{[}{]}{0pt}{}{n}{j}%
_{1/q}=%
\genfrac{[}{]}{0pt}{}{n}{j}%
_{q}q^{-j(n-j)}.
\]
Hence, for the left hand side of (\ref{exKM}), we have after changing $q$ to
$1/q$%
\begin{gather*}
\sum_{m=0}^{n}%
\genfrac{[}{]}{0pt}{}{n}{m}%
_{1/q}h_{m}(\cos(\theta+\varphi)|\frac{1}{q})h_{n-m}(\cos(\theta
-\varphi)|\frac{1}{q})\\
=\sum_{m=0}^{n}%
\genfrac{[}{]}{0pt}{}{n}{m}%
_{q}q^{-m(n-m)}(-1)^{m}q^{-\binom{m}{2}}\\
\times b_{m}(\cos(\theta+\varphi)|q)(-1)^{n-m}q^{-\binom{n-m}{2}}b_{n-m}%
(\cos(\theta-\varphi)|q)\\
=(-1)^{n}q^{-\binom{n}{2}}\sum_{m=0}^{n}%
\genfrac{[}{]}{0pt}{}{n}{m}%
_{q}b_{m}(\cos(\theta+\varphi)|q)b_{n-m}(\cos(\theta-\varphi)|q).
\end{gather*}
Now let us consider the right hand side of (\ref{exKM}) and change $q$ by
$1/q.$ We have
\begin{gather*}
\sum_{j=0}^{\left\lfloor n/2\right\rfloor }\frac{(q^{-1}|q^{-1})_{n}}%
{(q^{-1}|q^{-1})_{j}(q^{-1}|q^{-1})_{n-2j}}h_{n-2j}(x|q^{-1})h_{n-2j}%
(y|q^{-1})\\
=\sum_{j=0}^{\left\lfloor n/2\right\rfloor }\frac{(q)_{n}(-1)^{n}%
q^{-\binom{n+1}{2}}}{(q)_{j}(-1)^{j}q^{-\binom{j+1}{2}}(q)_{n-2j}%
(-1)^{n-2j}q^{-\binom{n-2j+1}{2}}}(-1)^{n-2j}\\
\times q^{-\binom{n-2j}{2}}b_{n-2j}(x|q)(-1)^{n-2j}q^{-\binom{n-2j}{2}%
}b_{n-2j}(y|q).
\end{gather*}
We deduce that (\ref{con1}) is true since we have $\binom{n}{2}+n\allowbreak
=\allowbreak\binom{n+1}{2}.$
\end{proof}

\end{document}